\documentclass[11pt]{article}
\usepackage{amsfonts,amsmath,amsthm,stmaryrd}
\usepackage{amssymb,amsmath,amsthm, enumerate}
\usepackage{latexsym}
\usepackage{fullpage}
\usepackage{tikz}
\usetikzlibrary[patterns]
\usetikzlibrary{arrows,shapes,trees, positioning} 
\usetikzlibrary{decorations.markings}
\tikzstyle{ann} = [fill=white,inner sep=1pt]

\usepackage{hyperref}
\usepackage{graphicx}
\usepackage{float}

\newtheorem{theorem}{Theorem}[section]

\newtheorem{prop}[theorem]{Proposition}
\newtheorem{question}{Question}[section]

\newtheorem{lemma}[theorem]{Lemma}
\newtheorem{remark}[theorem]{Remark}
\newtheorem{example}[theorem]{Example}
\newtheorem{cor}[theorem]{Corollary}

\newcommand{\one}{\mathbf 1}

\def\P{\mathbb{P}}
\def\E{\mathbb E}

\newcount \questionno
\def\question{\medbreak
        \global \advance \questionno 1
        \textbf{Problem \the\questionno}.\enspace \ignorespaces}
\newcommand{\remove}[1]{}
 \DeclareMathOperator{\var}{Var} 
  \DeclareMathOperator{\V}{Var} 
  \DeclareMathOperator{\Cov}{Cov}

\title{Convergence of the centered maximum of log-correlated Gaussian fields}

\author{Jian Ding\thanks{Partially supported by NSF grant DMS-1313596.} \\
University of Chicago\and Rishideep Roy\footnotemark[1] \\ University of Chicago \and Ofer Zeitouni\thanks{Partially supported by a grant from the 
	Israel Science Foundation and by the Herman
P. Taubman chair of Mathematics at the Weizmann Institute.
} \\ Weizmann Institute  \\\& New York University}
\date{March 10, 2015}
\begin{document}

\maketitle

\begin{abstract}
We show that the centered maximum of a sequence of log-correlated Gaussian 
fields
in any dimension converges in distribution, under the assumption that the 
covariances of the fields converge in a suitable sense. We identify the limit
as a randomly shifted Gumbel distribution, and characterize the random shift 
as the limit in distribution of a sequence of random variables, reminiscent of
the derivative martingale in the theory of Branching Random Walk and
Gaussian Chaos. We also discuss applications of the main convergence theorem 
and discuss examples that show that for logarithmically correlated fields, some
additional structural assumptions of the type we
make are needed for convergence 
of the centered maximum. 
\end{abstract}

\section{Introduction}

The convergence in law for the 
centered maximum of various log-correlated Gaussian fields, 
including branching Brownian motion (BBM), branching random walk (BRW), 
two-dimensional discrete Gaussian free field (DGFF), etc.,  
has recently been the focus of intensive study.
Of greatest relevance to the current paper are 
\cite{aidekon, Bramson83, BDZ13, LS87, Madaule13}. Historically,
the first result showing the correct centering and the tightness of the centered
maximum for BBM appears in the pioneering work \cite{Bramson78}, followed by the
proof of convergence of the law of the centered 
maximum \cite{Bramson83}; the latter
proof relied heavily on the F-KPP equation \cite{Fisher,KPP37}
describing the evolution of the
distribution of the maximum. A probabilistic description of the limit was
obtained in \cite{LS87}, using the notion of derivative martingale that 
they introduce. Convergence for the centered maximum of BRW 
with Gaussian increments was obtained in 
\cite{Bachmann}, while the analogous result
for general BRWs under mild assumptions was only obtained quite 
recently in the important work \cite{aidekon}, using the notion of 
derivative martingale to describe the limit; see also \cite{BDZ14}.

When no explicit tree structure is present, exact results concerning
the convergence in distribution of the maximum of Gaussian fields is
harder to establish. Recently, much progress has been achieved in
this direction: first, the 
two-dimensional DGFF was treated in \cite{BDZ13}, where convergence 
in distribution
of the centered maximum to a randomly shifted Gumbel random variable 
is established. The same result is obtained in \cite{Madaule13} for
a general class of log-correlated fields, the so called $*$-scale invariant 
models, where
the covariances of the fields admit a certain kernel representation. 
In the current paper, we extend in a systematic way the class
of logarithmically correlated fields for which the same result holds.
Our methods are inspired by \cite{BDZ13}, which in turn rely heavily on 
the modified second moment method, the modified BRW introduced in 
\cite{BZ10}, tail estimates proved for the DGFF in \cite{DZ12}, and Gaussian
comparisons.

We now introduce the class of fields considered in the paper.
Fix $d\in \mathbb{N}$ and let $V_N=\mathbb{Z}^d_N$ be the $d$-dimensional 
box of side length $N$ with the left bottom corner located at the origin.  
For convenience, we consider a suitably normalized version of 
Gaussian fields $\{\varphi_{N, v}: v\in V_N\}$ satisfying the following.
\begin{itemize}
	\item[(A.0)] \textbf{(Logarithmically  bounded fields)} \
There 
exists a constant $\alpha_0 >0$ such that for all $u, v\in V_N$,\\
 $$\var \varphi_{N, v} \leq \log N + \alpha_0$$ and 
$$\E (\varphi_{N, v} - \varphi_{N, u})^2  \leq 2 \log_+ |u-v| - | \var \varphi_{N, v} - \var \varphi_{N, u}| + 4\alpha_0,$$
\end{itemize}
where $|\cdot|$ denotes the Euclidean norm and 
$\log_+x = \log x \vee 0$. Note that Assumption (A.0) is rather mild 
and in particular is satisfied by the two-dimensional DGFF and $*$-scale invariant models. It is however strong enough to provide 
an a-priori tight estimate on the right tail of the distribution of the maximum.

Set $M_N = \max_{v\in V_N} \varphi_{N, v}$ and
\begin{equation}\label{eq-thm-expectation}
m_N = \sqrt{2d} \log N - \tfrac{3}{2\sqrt{2d}} \log \log N \,.
\end{equation} 
\begin{prop}\label{prop-maxrtfluc} 
Under Assumption (A.0), there exists a
constant $C = C(\alpha_0)> 0$ 
such that for all $N \in \mathbb{N}$ and $z \ge 1$,
\begin{equation}\label{eq-maxrtflucrttail}
\P ( M_N \ge m_N+z)
 \leq 
C ze^{- \sqrt{2d} z} e^{- C^{-1} z^2/n} \,.
\end{equation}
Furthermore, for all  $z \ge 1, y \ge 0$ and $A \subseteq V_N$ we have
\begin{equation}\label{eq-maxrtflucrttail2}
 \P( \max_{v \in A} \varphi_{N, v} \ge m_N + z -y ) \leq C \left( \tfrac{|A|}{|V_N|} \right)^{1/2} ze^{-\sqrt{2d}(z-y)} \,.
\end{equation}
\end{prop}
\noindent The proof of Proposition \ref{prop-maxrtfluc} is provided in Section 
\ref{sec:tightness}.

By Proposition~\ref{prop-maxrtfluc}, 
if one has a complementary lower bound showing that for a large enough
constant $C$,
$\max_{v\in V_N} \varphi_{N, v}>m_N-C$
with high probability,
it follows that 
the maximizer of the Gaussian field is away from the boundary with high
probability. Therefore,  in the study of convergence of the centered maximum,
it suffices to consider the Gaussian field away from the boundary 
(more precisely, with distance $\delta N$ away from the boundary where 
$\delta\to 0$ after $N\to \infty$).  
In light of this, introduce the sets 
$V_N^\delta=\{z\in V_N: d(z,\partial V_N)\geq \delta N\}$ and
$V^\delta=[\delta,1-\delta]^d$, where $d(z,\partial V_N)=\min\{\|z-y\|_\infty: y\not\in V_N\}$.
Then, introduce the  following assumption.
\begin{itemize}
	\item[(A.1)] \textbf{(Logarithmically correlated fields)} \
For any $\delta>0$ there 
exists a constant $\alpha^{(\delta)}>0$ such that for all $u, v\in V_N^\delta$,
		$|\Cov( \varphi_{N, v},\varphi_{N, u}) - (\log N - \log_+ |u-v| 
		)| \leq \alpha^{(\delta)}$.
\end{itemize}
We do not assume Assumption (A.1) for $\delta = 0$ since we wish to 
incorporate Gaussian fields with Dirichlet boundary conditions, such as the
two dimensional DGFF.

Assumptions (A.0) and (A.1) are enough to ensure the tightness of the
sequence $\{M_N-m_N\}_N$.
\begin{theorem}\label{thm-tightness}
Under Assumptions (A.0) and (A.1), we have that $\E M_N = m_N + O(1)$ where
the $O(1)$ term depends on $\alpha_0$ and $\alpha^{(1/10)}$.
In addition, the sequence $M_N - \E  M_N$ is tight.
\end{theorem}
\noindent
(The constant  $1/10$ in Theorem \ref{thm-tightness} 
could be replaced by any positive number that is less than $1/3$.) 
The proof of Theorem \ref{thm-tightness} is provided in 
in Section \ref{sec:tightness}.

As we will explain later, Assumptions (A.0) and (A.1) on their own cannot 
ensure convergence in law for the centered maximum. To ensure the latter
we introduce the following additional assumptions.
First, we assume convergence of the covariance in finite scale around the 
diagonal.
\begin{itemize}
	\item[(A.2)] \textbf{(Near diagonal behavior)} \
		There exist a continuous function $f : (0,1)^d \mapsto 
	\mathbb{R}$ and a function $g: \mathbb Z^d \times \mathbb Z^d 
	\mapsto \mathbb R$ such that the following holds. For all $L, 
	\epsilon,\delta > 0$, 
	there exists $N_0= N_0(\epsilon, \delta, L)$ such that for all  
	$x \in V^\delta$, $u, v \in [0,L]^d $ and $N\geq N_0$ 
	we have
\begin{equation*}|\Cov( \varphi_{N, xN+v},\varphi_{N, xN+u})-
	\log N -f(x) -g(u,v)|< \epsilon\,.
\end{equation*}
\end{itemize}
\noindent
Next, we introduce an assumption concerning 
convergence of the covariance off-diagonal 
(in a macroscopic scale). Let 
$\mathcal {D}^d=\{(x,y): x,y\in (0,1)^d, x\neq y\}$.
\begin{itemize}
\item[(A.3)] \textbf{(Off diagonal behavior)}
	There exists a continuous function $h: \mathcal{D}^d\mapsto \mathbb{R}$
	such that the following holds. For all $L, \epsilon,\delta > 0$, 
	there exists $ N_1= N_1(\epsilon, \delta, L)>0$ such that for all 
	$x, y\in V^\delta$ with $|x-y| \ge \frac{1}{L}$ and $N \ge N_1 $ we have
$$
|\Cov( \varphi_{N, xN},\varphi_{N, yN})-h(x,y)|< \epsilon\,.
$$
\end{itemize}
Assumptions (A.2) and (A.3) control the convergence of the covariance on 
both microscopic and macroscopic scale, but allows for fluctuations of order $1$
in mesoscopic scale. It is not hard to check that both the DGFF and the 
$*$-scale fields satisfy Assumptions (A.0)--(A.3). A further example will be
developed in Section \ref{sec:example}.

We remark that Assumptions (A.2) and (A.3) are not necessary for the 
convergence of the centered maximum. Indeed,
one can violate Assumptions (A.2) and (A.3) by perturbing the field at a 
single vertex, but this would not affect the convergence in law of 
the centered maximum,  since with overwhelming probability, 
the maximizer is not at the perturbed vertex. 
However, if Assumptions (A.2) and (A.3) are violated ``systematically'', 
one should not expect a convergence in law for the centered maximum. We will 
give two examples at the end of the introduction as a demonstration on how 
violating (A.2) or (A.3) could destroy convergence in 
law for the centered maximum.

Our main result is the following theorem.
\begin{theorem}\label{thm-convergence}
Under Assumptions (A.0), (A.1), (A.2) and (A.3), the
sequence $\{M_N - \E M_N\}_N$ converges in distribution. 
 \end{theorem}
As a byproduct of our proof, we also characterize the limiting law of 
$(M_N - m_N)$ as a Gumbel distribution with random shift, given by a positive 
random variable $\mathcal Z$
which is the weak limit of a sequence of a sequence 
${\mathcal Z}_N$, defined as
\begin{equation}\label{eq-def-derivative-martingale}
\mathcal Z_N = \sum_{v\in V_N} (\sqrt{2d} \log N - \varphi_{N, v}) e^{-\sqrt{2d}(\sqrt{2d} \log N - \varphi_{N, v})}\,.
\end{equation}
In the case of BBM, the corresponding 
sequence ${\mathcal Z}_N$ is precisely 
the derivative martingale, introduced in \cite{LS87}.
It also occurs in the case of BRW, see
\cite{aidekon}, and plays a similar role in the study of
critical Gaussian multiplicative chaos
\cite{SDals}. Even though in our case the sequence ${\mathcal Z}_N$
is not necessarily a martingale, in analogy with these previous 
situations we keep refering to it as \textit{the derivative martingale}. 
The definition naturally extends to a derivative martingale 
\textit{measure} on $V_N$ by setting, for $A\subset V_N$,
$$\mathcal Z_{N, A} = \sum_{v\in A} (\sqrt{2d} \log N - 
\varphi_{N, v}) e^{-\sqrt{2d}(\sqrt{2d} \log N - \varphi_{N, v})}.$$ 
\begin{theorem}\label{explicit-limit-law}
Suppose that Assumptions (A.0), (A.1), (A.2) and (A.3) hold. Then
the derivative martingale $\mathcal Z_N$ converges in law to a positive 
random variable $\mathcal Z$. In addition, the limiting law 
$\mu_\infty$ of $M_N - m_N$ can be expressed by 
$$\mu_\infty((-\infty, x]) = \E e^{-\beta^* \mathcal Z e^{-\sqrt{2d} x}}\,, \mbox{ for all } x\in \mathbb R\,,$$
where $\beta^*$ is a positive constant.
\end{theorem}
\begin{remark}
In \cite{BL13}, \cite{BL14},
the authors used the convergence of the centered maximum,
a-priori information on the geometric properties of the
clusters of  
near-maxima of the DGFF and a beautiful invariance argument and derived
the convergence in law of the process of near extrema of
the two-dimensional DGFF, and its properties.
A natural extension of our work would be to study the
extremal process in the class of processes studied here, and tie it
to properties of the derivative martingale measure.
\end{remark}

\begin{remark}
  \label{rem-universal}
  Our proof will show that the random variable $\mathcal Z$ appearing in
  Theorem \ref{explicit-limit-law} depends only on the functions
  $f(x),h(x,y)$ appearing in Assumptions (A.2) and (A.3), while the
  constant $\beta^*$ depends on other parameters as well. 
  In particular, two sequences of fields that differ only at the
  microscopic level will have the same limit law for their centered
  maxima, up to a (deterministic) shift.
  We provide more details at the end of Section \ref{sec:approx}.
\end{remark}

\medskip

\noindent{\bf A word on proof strategy.} This paper
is closely related to \cite{BDZ13}, which
dealt with  2D GFF.
The proof in \cite{BDZ13} consists of three main steps: 
\begin{enumerate}
	\item
Decompose the DGFF to a sum of a coarse field and a fine field 
(which itself is a DGFF), and further approximate the fine field 
as a sum of modified branching random walk (see Section~\ref{sec:MBRW} 
for definition) and a local DGFF. It is crucial for the proof 
that the different 
components are independent of each other, and that the approximation
error is small enough so that the value of the maximum is not altered
significantly. These approximations were 
constructed using heavily the Markov field property of DGFF, and
detailed
estimates for the Green function of random walk.  
\item  Use a modified second moment method in order to compute the asymptotics 
	of the right  tail for the distribution of the maximum of
	the fine field,  as well as derive a limiting distribution 
	for the location of the maximizer in the fine field. 
\item Combine the limiting right tail estimates for the maximum of the
	fine field and the behavior of the coarse field to 
	deduce the convergence in law.  
	\end{enumerate}

In the general setup of logarithmically correlated fields,
it is not a priori clear how can one
decompose the field
by an (independent)
sum of a coarse field, an MBRW and a local field, as the Markov field 
property is no longer available. A natural approach under 
our assumptions
is to employ the self-similarity of the fields, and to approximate 
the coarse and local fields by an instance of $\{\varphi_{K, v}:  v\in V_K\}$ 
for some $K \ll N$. One difficulty in this attempt is to control the error 
of the approximation and its influence on the law of the maximum. 
In order to address this issue, we partition  the box $V_N$ to sub-boxes 
congruent to $V_L$, and borrow a key
idea from \cite{BL13} to show that the law of the 
maximum of a log-correlated fields has the following 
invariance property:
if one adds i.i.d.\ Gaussian variables with 
variance $O(1)$ to each sub-box of the field (here the same variable 
will be added to each vertex in the same sub-box), where the 
size $L$ of the sub-box is either $K$ or $N/K$ (assuming $K$ grows to 
infinity arbitrarily slow in $N$), then the law of the maximum for 
the perturbed field is simply a shift of the original law where the 
shift can be explicitly determined (see Lemma~\ref{lem-nfmax}). In light of 
this, in Subsection~\ref{sec:decomposition} we approximate the 
field $\{\varphi_{N, v}\}$ by 
the 
sum of coarse field (which is given by 
$\{\varphi_{KL, v}: v\in V_{KL}\}$), an MBRW, and a local field
(which is
given by independent copies of $\{\varphi_{K'L', v}: v\in V_{K'L'} \}$)
(here the parameters satisfy $N\gg K'\gg L'\gg K\gg L$). In this construction, 
we make sure that the error in the covariance between two vertices 
is $o(1)$ if their distance is not in between $L$ and $N/L'$, and the error is 
$O(1)$ otherwise. Then we apply Lemma~\ref{lem-nfmax} 
(and Lemma~\ref{lem-epsilon-domination}) to argue that our approximation
indeed recovers the law of the maximum for the original field. In 
Subsection~\ref{sec:limit-law}, we present the proof for the convergence in law
for the centered maximum of the approximated field we constructed and, as
in \cite{BDZ13},
it readily also yields the convergence in distribution for 
the derivative martingale constructed from the original field.  

As in the case of the DGFF in two dimensions,
a number of properties for the log-correlated fields are needed, and are
proved by adpating or modifying the arguments used in that case.
Those properties are:
\begin{enumerate}
	\item 
The tightness of $M_N - m_N$, and the
bounds on the right and left tails of $M_N - m_N$ as well as 
certain geometric properties of maxima for the log-correlated fields 
under consideration,  follow from modifying arguments in 
\cite{BZ10, DZ12, Ding11}. This is explained in Section~\ref{sec:tightness}.
\item Precise asymptotics for the rigth tail of the distribution of the maximum 
	of the fine field follow from arguments similar to
	\cite{BDZ13} with a number of simplifications,
	as our fine field has a nicer structure than its analogue 
	in \cite{BDZ13}, whereas 
	the coarse field employed in this paper 
	is constant over each box;  in particular,
	there is no need to consider the distribution for the location 
	of the maximizer in the fine field as done in \cite{BDZ13}. 
	The adaption is explained in the Appendix.
	\end{enumerate}

\noindent {\bf The role of Assumptions (A.2) and (A.3).}  We next construct
two examples that demonstrate that one cannot totally dispense of 
Assumptions (A.2) and (A.3). Since the examples are only ancillary to our main 
result, we will give only give a brief sketch for the verification of 
the claims made concerning these examples.

\begin{example} 
Fix $d=2$ and let $\{\varphi_{N, v}: v\in V_N\}$ be the DGFF on $V_N$ 
(normalized so that it satisfies Assumptions 
(A.0), (A.1), (A.2) and (A.3)), with
$Z_N=\max_{v\in V_N} \varphi_{N,v}$. 
Let $V_{N, 1}$ and $V_{N,2}$ be the left and right halves of the box $V_N$.
Let $\{\epsilon_{N, v}: v\in V_N\}$ and $X$ 
be i.i.d.\  standard Gaussian variables. 
Define 
$$\tilde \varphi_{N, v} = \left\{\begin{array}{ll}
	\varphi_{N, v} + \sigma X + \epsilon_{N, v}, &
	v\in V_{N, 1}\\
	 \varphi_{N, v}, & v\in V_{N, 2}
 \end{array}\right. , 
 \quad
\hat \varphi_{N, v} = \left\{\begin{array}{ll}
	\varphi_{N, v} + \sigma X , &
	v\in V_{N, 1}\\
	\varphi_{N, v}+\sigma_N'\epsilon_{N,v}, & v\in V_{N, 2}
 \end{array}\right. .
 $$
Set $\tilde M_N =  \max_{v\in V_N} \tilde \varphi_{N, v}$ and 
$\hat M_N =  \max_{v\in V_N} \hat \varphi_{N, v}$. We first
claim that there exist $\sigma'_N $ depending on $(N, \sigma)$ but bounded 
from above by an absolute constant such that $\E \tilde M_N = \E \hat M_N$.  
In order to see that, note that,  by Theorem~\ref{thm-tightness},
$$\E \tilde M_N \leq \E \max_{v\in V_{N/2}}  
 \varphi_{N, v} + \sigma \E \max(0, X) \leq 2 \log N - \tfrac{3}{4} \log \log N + \sigma \E \max(0, X)+O(1)\,,$$
where $O(1)$ is an error term independent of all parameters. n addition, 
by considering a $N/2$-box in the left side and dividing 
the right half box into two copies of $N/2$-boxes, one gets that
$$\E \hat M_N \geq \E \max (Z_{N/2} + \sigma X, Z'_{N/2} + \sigma'_N \epsilon', Z''_{N/2} + 
\sigma'_N \epsilon'') \geq \E Z_{N/2} + \frac12 \sigma'_N 
\E\max(\epsilon',\epsilon'')+ \sigma \E X \mathbf 1_{X\geq 0}\,.$$
where $Z_{N/2}, Z'_{N/2}, Z''_{N/2}$ are three independent
copies with law $\max_{v\in V_{N/2} }\varphi_{N, v}$ and 
$\epsilon' = \epsilon_{N, v_*^1}, \epsilon'' = \epsilon''_{N, v^*_2}$ 
(here $v_1^*$ and $v_2^*$ are the maximizers of the DGFF in the two 
$N/2$-boxes on the right half of $V_N$, respectively). The claim follows from
combining the last two displays.

Now, choose $\sigma$ to be a large fixed constant so that for
$0<\lambda < \log \log N$, 
\begin{eqnarray}
	\label{eq-paris1}
\P(\tilde M_N  \geq \E  Z_N+ \lambda) & \geq &
\P(\max_{v\in V_{N,1}} \{ \varphi_{N, v} + \sigma X + \epsilon_{N, v}\}  
\geq  \E Z_N + \lambda) \nonumber \\
& \geq &
\P((1 + 1/4\log N ) \max_{v\in V_{N, 1}} \{\varphi_{N, v} + \sigma X\} 
\geq \E Z_N + \lambda)\nonumber \\
&\geq &\P(\max_{v\in V_{N, 1}} \varphi_{N, v} + 
\sigma X \geq \E  Z_N +\lambda- 1/10)\,.
\end{eqnarray}
(Here, 
the second inequality is due to Slepian's comparison lemma 
(Lemma~\ref{lem-slepian}) and the fact that $\sigma$ is large, while the last
inequality uses that $2/(1+1/(4\log N))\leq 2-(\log N)/10$ for $N$ large.)
Further,
\begin{eqnarray}
	\label{eq-paris2}
\P(\hat M_N  \geq \E  Z_N + \lambda) & \leq &
\P(\max_{v\in V_{N, 1}} \varphi_{N, v}+ \sigma X   \geq  \E Z_N + \lambda) 
+ \P(\max_{v\in V_{N, 2}} \varphi_{N, v} + \epsilon'_{N, v} \geq  \E Z_N +
\lambda)\nonumber\\
&\leq &  \P(\max_{v\in V_{N, 1}} \varphi_{N, v}+ \sigma X   
\geq  \E Z_N + \lambda)  + O(1) \lambda e^{-2\lambda}\,,
\end{eqnarray}
where the last inequality follows from Proposition~\ref{prop-maxrtfluc}. 
Combining \eqref{eq-paris1} and \eqref{eq-paris2}
and using the form of the
limiting right tail of the two-dimensional DGFF as in 
\cite[Proposition 4.1]{BDZ13}, one obtains 
that for $\lambda, \sigma$ sufficiently large but independent of $N$,
$$\limsup_{N\to \infty}\P(\tilde M_N  \geq \E  Z_N+ \lambda) \geq 
(1 + c) \limsup_{N\to \infty}\P(\hat M_N  \geq \E  Z_N + \lambda)  
\geq c(\sigma) \lambda e^{-2\lambda}\,,$$
where $c>0$ is an absolute constant and $c(\sigma)$ satisfies 
$c(\sigma) \to _{\sigma\to \infty} \infty$. This implies that the 
laws of $\tilde M_N-EM_N$ and $\hat M_N-E\hat M_N$ do
not coincide in the limit  $N\to \infty$.

 Finally,  set $\bar \varphi_{N, v} = \tilde \varphi_{N, v}$  
 for all $v\in V_{N}$ and odd $N$, and $\bar \varphi_{N, v} = 
 \hat \varphi_{N, v}$  for all $v\in V_{N}$ and even $N$. 
One then sees that
the sequence of Gaussian fields $\{\bar \varphi_{N, v}: v\in V_N\}$ 
satisfies Assumptions (A.0), (A.1) and (A.3) (while not satisfying
(A.2)), but the law of the centered maximum does not converge. 
\end{example}

\begin{example}
Let $\{\varphi_{N, v}: v\in V_N\}$ be
a sequence of Gaussian fields satisfying (A.0), (A.1) and (A.2),
such that the law of the centered maximum converges. 
Consider the fields $\{\tilde \varphi_{N, v}: v\in V_N\}$ where 
$\tilde \varphi_{N, v} = \varphi_{N, v} + \mathbf 1_{N \mbox{ is even}} X_N$ 
with $X_N$ a sequence of i.i.d. standard Gaussian variables. Then,
the field
$\{\tilde \varphi_{N, v}: v\in V_N\}$ satisfies (A.0), (A.1) and (A.2) 
(possibly increasing the values of $\alpha^{(\delta)}$ by 
1 for all $0\leq \delta\leq 1$). 
However, the centered law of  the maximum of 
$\{\tilde \varphi_{N, v}: v\in V_N\}$ cannot converge. 
\end{example}

\section{Expectation and tightness for the maximum} \label{sec:tightness}
This section is devoted to the proofs of 
Proposition~\ref{prop-maxrtfluc} and 
Theorem~\ref{thm-tightness}, and to an auxiliary lower bound on the right
tail of the distribution of the maximum, see
Lemma \ref{lem-righttail}.
The proof of the proposition is
very similar to the proof in the case of the DGFF in dimension two, 
using a comparison with an appropriate BRW; Essentially, the proposition gives 
the correct right tail behavior of the distribution of the
maximum. In contrast, given the 
proposition, in order to prove 
Theorem~\ref{thm-tightness}, one needs an 
upper bound on the \textit{left} tail of that distribution.
In the generality of this work, one cannot hope for 
a universal sharp estimate on the left tail, as witnessed
by the drastically different left tails exhibited in the cases
of the modified branching random walk and the two-dimensional DGFF, see
\cite{Ding11}. We will however provide the following 
universal \textit{upper} bound 
for the decay of the left tail.
\begin{lemma}\label{lem-GFlefttail}
Under Assumption (A.1) there exist constants $C, c >0$ 
(depending only on $\alpha_{1/10}, d$) so that for all $n \in \mathbb{N}$ 
and $ 0 \leqslant \lambda \leqslant (\log n)^{2/3}$, 
$$ \P (\max_{v \in V_N} \varphi_{N,v} \leqslant  m_N - \lambda ) \leqslant Ce^{-c\lambda}\,.$$
\end{lemma} 
\noindent
Theorem \ref{thm-tightness} follows at once from
Proposition~\ref{prop-maxrtfluc} and Lemma~\ref{lem-GFlefttail}.

Later, we will need the following complimentary lower bound on the right tail.
\begin{lemma}\label{lem-righttail} Under Assumption (A.1), there exists a 
	constant $C>0$ depending only on $(\alpha_0, \alpha^{(1/10)}, d)$ 
	such that for all $\lambda \in [1,\sqrt{\log N}]$,
\begin{equation*}\P(M_N > m_N + \lambda ) \geq C^{-1}\lambda e^{-\sqrt{2d}\lambda}\,.
\end{equation*}
\end{lemma}

 \subsection{Branching random walk and modified branching random walk}\label{sec:MBRW}
 
The study of extrema for 
log-correlated Gaussian fields is possible because they exhibit an approximate
tree structure and can be efficiently compared with branching random walk and 
the modified branching random walk  
introduced in \cite{BZ10}. In this subsection,
we briefly review the definitions of BRW and MBRW in $\mathbb Z^d$.  
 
Suppose $N=2^n$ for some $n \in \mathbb{N}$. 
For $j=0,1,\ldots, n$, define $\mathcal{B}_j$ to be the
set of $d$-dimensional cubes of side length $2^j$ with corners in 
$\mathbb{Z}^d$. Define $\mathcal{BD}_j$ to be those elements of 
$\mathcal{B}_j$ which are of the form $\left( [0,2^j -1]\cap \mathbb{Z} 
\right)^d + (i_1 2^j, i_2 2^j,\ldots, i_d 2^j )$, where 
$i_1, i_2, \ldots, i_d$ are integers. 
For $x \in V_N$, define $\mathcal{B}_j (x)$ to be those 
elements of $\mathcal{B}_j$ which contains $x$. Define 
$\mathcal{BD}_j (x)$ similarly. 

Let $\{ a_{j,B} \}_{j \ge 0, B \in \mathcal{BD}_j}$ be
a family of i.i.d.\ Gaussian variables of variance $\log 2$. Define the
\textit{branching random walk} (BRW) $\{ \mathcal{R}_{N,z} \}_{z \in V_N}$ by 
$$ \mathcal{R}_{N, z} = \sum_{j=0}^n \sum_{B \in \mathcal{BD}_j (z)} a_{j,B}\,,
\quad
 z\in V_N\,.$$ 

Let $\mathcal{B}_j^N$  be the subset of $\mathcal{B}_j$ consisting of 
elements of the latter with lower left corner in $V_N$. 
Let $\{ b_{j,B} : j \ge 0, B \in \mathcal{B}_j^N\}$ be 
a family of independent Gaussian variables such that 
$ \var b_{j,  B} = \log 2\cdot  2^{-dj}$ for all 
$B\in \mathcal B_j^N$. Write $B\sim_N B'$ if $B = B'+ (i_1 N, \ldots, i_d N)$ 
for some integers $i_1, \ldots, i_d \in \mathbb Z$. Let 
\begin{eqnarray*}
 b_{j,B}^N = \begin{cases}
b_{j,B} \qquad  B \in \mathcal{B}_j^N\,,\\
b_{j,B'} \qquad B \sim_N B' \in \mathcal{B}_j^N\,.
\end{cases}
\end{eqnarray*}
Define the \textit{modified branching random walk} (MBRW) 
$\{ \mathcal{S}_{N,z} \}_{z \in V_N}$ by
\begin{equation}\label{eq-def-MBRW} 
\mathcal{S}_{N, z} = \sum_{j=0}^n \sum_{B \in \mathcal{B}_j (z)} b_{j,B}^N\,,\quad
z\in V_N\,.
\end{equation}
The proof of the following lemma is an straightforward adaption of 
\cite[Lemma 2.2]{BZ10} for dimension $d$, which we omit.
\begin{lemma}\label{lem-covapp} There exists a constant $C$ depending only on $d$ such that for $N=2^n$ and $x,y \in V_N$
\begin{equation*}\label{eq-covapp}|\Cov(\mathcal{S}_{N,x}, \mathcal S_{N,y})-(\log N- \log( |x-y|_N \vee 1))| \le C\,, \end{equation*}
where $|x-y|_N=\min_{y'\sim_N y} |x-y'|$.
\end{lemma}

In the rest of the paper, we assume that the constants  
$\alpha_0, \alpha^{(\delta)}$ in Assumptions (A.0) and (A.1) are taken
large enough so that the MBRW satisfies the assumptions.

\subsection{Comparison of right tails}
The following Slepian's comparison lemma for Gaussian processes 
\cite{slepian62} will be useful.
\begin{lemma}\label{lem-slepian}
Let $\mathcal{A}$ be an arbitrary finite index set and let $\{ X_a : a \in \mathcal{A}\}$ and $\{ Y_a : a \in \mathcal{A}\}$ be two centered Gaussian processes such that: $\E(X_a - X_b)^2 \ge \E(Y_a - Y_b)^2$, for all $a,b \in \mathcal{A}$ and $\V(X_a) = \V (Y_a)$ for all $a \in \mathcal{A}$. Then $\P(\max_{a \in \mathcal{A}} X_a \ge \lambda) \ge \P(\max_{a \in \mathcal{A}} Y_a \ge \lambda)$ for all $\lambda \in \mathbb{R}$.
\end{lemma}
The next lemma compares the right tail for the maximum of $\{\varphi_{N, v}: v\in V_N\}$ to that of a BRW.
\begin{lemma}\label{lem-righttailcomp} Under Assumption (A.0), there exists an integer $\kappa = \kappa(\alpha_0) > 0$ such that for all $N$ and $\lambda \in \mathbb{R}$ and any subset $A\subseteq V_N$
\begin{equation}\label{eq-righttailcompBRW} 
\P(\max_{v \in A} \varphi_{N, v} \ge \lambda) \le  2\P(\max_{v \in 2^{\kappa} A }  \mathcal{R}_{2^{\kappa}N, v}\ge \lambda )\,.
 \end{equation}
\end{lemma}
\begin{proof}   For $\kappa\in \mathbb N$, consider the map 
\begin{equation}\label{eq-def-psi}
\psi_N = \psi_N^{(\kappa)} : V\mapsto 2^\kappa V\mbox{ such that }\psi_N(v) = 2^\kappa v\,.
\end{equation} 
By Assumption (A.0), we can choose a sufficiently large $\kappa$ depending on $\alpha_0$ such that $\V(\varphi_{N, v})\le \V(\mathcal{R}_{2^\kappa N, \psi_N(v)})$ for all $v  \in V_N$. So, we can choose a collection of positive numbers 
$$a_v^2 = \var \mathcal R_{2^\kappa N, \psi_N(v)} - \var\varphi_{N, v}\,,$$ such that $\V(\varphi_{N,v}+a_v X)= \V(\mathcal{R}_{2^\kappa N, \psi_N(v)})$ for all $v  \in V_N$, where $X$ is  a standard Gaussian random variable, independent of everything else. Since the BRW has constant variance over all vertices, we get that 
$$\E(\varphi_{N,u}+a_u X-\varphi_{N,v}-a_v X)^2 \le \E(\varphi_{N,u}-\varphi_{N,v})^2 + (a_v - a_u)^2 \leq  \E(\varphi_{N,u}-\varphi_{N,v})^2 + |\var \varphi_{N, v} - \var\varphi_{N, u}| \,.$$
Combined with Assumption (A.0), it yields that
$$\E(\varphi_{N,u}+a_u X-\varphi_{N,v}-a_v X)^2  \leq  2 \log_+|u-v| + 4 \alpha_0\,.$$
Since $\E(\mathcal{R}_{2^\kappa N, \psi_N(u)}-\mathcal{R}_{2^\kappa N, \psi_N(v)})^2 - 2 \log_+ |u-v| \geq \log 2 \kappa - C_0$ (where $C_0$ is an absolute constant),  we can choose sufficiently large $\kappa$ depending only on $\alpha_0$ such that 
$$\E(\varphi_{N,u}+a_u X-\varphi_{N,v}-a_v X)^2 \le 
\E(\mathcal{R}_{2^\kappa N, \psi_N(u)}-\mathcal{R}_{2^\kappa N, \psi_N(v)})^2\,, \mbox{ for all } u, v \in V_N\,.$$
Combined with Lemma~\ref{lem-slepian}, it gives that for all $\lambda \in \mathbb{R}$ and $A\subseteq V_N$
$$\P(\max_{v \in A} \varphi_{N,v} +a_v X \ge \lambda ) \le \P(\max_{v \in A} \mathcal{R}_{2^\kappa N, \psi_N(v)} \ge \lambda)\,.$$
In addition, by independence and symmetry of $X$ we have
\begin{equation*}
\P(\max_{v \in A} \varphi_{N, v} +a_v X \ge \lambda ) \ge \P(\max_{v \in A} \varphi_{N,v} \ge \lambda, X\ge 0 )  
= \tfrac{1}{2} \P(\max_{v \in A} \varphi_{N, v} \ge \lambda )  \,.
\end{equation*}
This completes the proof of the desired  bound. 
\end{proof}

\begin{proof}[Proof of Proposition~\ref{prop-maxrtfluc}]
An analogous statement was proved in \cite[Lemma 3.8]{BDZ13} for the case of 
2D DGFF. In the proof of \cite[Lemma 3.8]{BDZ13}, the desired inequality was 
first proved for BRW on the 2D lattice and then deduced for 
2D DGFF applying \cite[Lemm 2.6]{DZ12}, which is the analogue of Lemma 
\ref{lem-righttailcomp} above. The argument for BRW in \cite[Lemma 3.8]{BDZ13} 
carries out (essentially with no change) from dimension two to dimension $d$. 
Given that, an application of Lemma~\ref{lem-righttailcomp} 
completes the proof of the proposition.
\end{proof}

A complimentary lower bound on the right tail is also available.
\begin{lemma}\label{lem-compare-righttail-lower}
Under Assumption (A.1), there exists an integer $\kappa = \kappa(\alpha^{(1/10)})  > 0$ such that for all $N$ and $\lambda \in \mathbb{R}$ 
\begin{equation}\label{eq-righttailcomp}  \P(\max_{v \in V_N} \varphi_v^N \ge \lambda) \geq  \tfrac{1}{2} \P( \max_{v \in V_{2^{-\kappa}N} } \mathcal{S}_{2^{-\kappa}N, v}\ge \lambda )\,.
 \end{equation}
\end{lemma}
\begin{proof}  It suffices to consider $M^{(1/10)}_N = \max_{v\in V_N^{1/10}} 
	\varphi_{N, v}$. By Assumption (A.1) and an argument analogous
	to that used in the proof of Lemma~\ref{lem-righttailcomp} 
	(which can be raced back to the proof of \cite[Lemma 2.6]{DZ12}), one
	deduces that for $\kappa = \kappa(\alpha^{(1/10)})$, 
$$\P(M^{(1/10)}_N \geq \lambda) \geq  \tfrac{1}{2} \P( \max_{v 
	\in V_{2^{-\kappa}N} } \mathcal{S}_{2^{-\kappa}N, v}\ge \lambda ) 
	\mbox{ for all } \lambda \in \mathbb R\,.$$
This completes the proof of the lemma.
\end{proof}
We also need the 
following estimate on the right tail for MBRW in $d$-dimension. 
The proof is a routine
adaption of the proof of \cite[Lemma 3.7]{DZ12} to arbitrary dimension, and
is omitted.
\begin{lemma}\label{lem-righttailMBRW}
There exists an absolute constant $C>0$ such that for all $\lambda\in [1, \sqrt{\log n}]$, we have
\begin{equation*}C^{-1}\lambda e^{-\sqrt{2d}\lambda} \le \P(\max_{v\in V_N} \mathcal S_{N, v} > m_N + \lambda ) \le C\lambda e^{-\sqrt{2d}\lambda}\,.
\end{equation*}
\end{lemma}
\begin{proof}[Proof of Lemma~\ref{lem-righttail}]
Combine Lemma~\ref{lem-compare-righttail-lower} and Lemma~\ref{lem-righttailMBRW}.
\end{proof}

\subsection{An upper bound on the left tail}

This subsection is devoted to the proof of Lemma~\ref{lem-GFlefttail}. The proof
consists of two steps: (1) a derivation of an exponential  upper bound on the 
left tail for the MBRW; (2) a comparison
of the left tail for  general log-correlated Gaussian field to that of the MBRW. 
\begin{lemma}\label{lem-MBRWlefttail}
There exist constants $C, c >0 $ so that for all $n \in \mathbb{N}$ and 
$ 0 \leqslant \lambda \leqslant (\log n)^{2/3}$, 
$$\P (\max_{v \in V_N} \mathcal{S}_{N,v} \leqslant m_N - \lambda ) \leqslant Ce^{-c\lambda}\,.$$
\end{lemma}
\begin{proof}
A trivial extension of the
arguments in \cite{BZ10} (for the MBRW in dimension two) yields the tightness 
for the maximum of the MBRW in dimension $d$ arounds its expectation, with 
the latter 
given by \eqref{eq-thm-expectation}. Therefore, there exist  
constants $\kappa , \beta> 0$ such that for all $N  \geq 4$, 
\begin{align}
\P(\max_{v \in V_N} \mathcal{S}_{N,v} \geqslant m_N - \beta ) 
\geqslant 1/2\,. \label{eq-lefttaillb}
\end{align}
In addition, a simple calculation gives that for all $N\geq N'\geq 4$ (adjusting
the value of $\kappa$ if necessary),
\begin{equation}
\sqrt{2d} \log (N/N') - \frac{3}{4d} \log(\log N/ \log N') - 
\kappa \leqslant m_N - m_{N'} \leqslant \sqrt{2d} \log (N/N') + 
\kappa \,. \label{eq-mean-difference}
\end{equation}
Let $\lambda'=\lambda/2$ and $N'=N\exp(-\frac{1}{\sqrt{2d}}(\lambda'- \beta - 
\kappa - 4))$. By \eqref{eq-mean-difference}, one has $m_N - m_{N'} 
\leqslant \lambda' - \beta$. Divide $V_N$  into disjoint boxes of side 
length $N'$, and consider a maximal collection $\mathcal B$ of $N'$-boxes 
such that all the pairwise distances are
at least $2N'$, implying that
 $|\mathcal B| \geq \exp(\frac{\sqrt{d}}{\sqrt{2}}(\lambda'- \beta - \kappa - 8- 4\sqrt{d}))$.  Now consider the  modified MBRW
$$\tilde{\mathcal{S}}_{N, v} = g_{N', v} + \phi \quad\forall v \in B \in \mathcal{B}\,,$$
where 
$\phi$ is an zero mean Gaussian variable with variance  $\log (N/N')$ 
and $\{ g_{N', v} : v \in B \}_B$ are the MBRWs defined on the boxes $B$,
independently of each other and  of $\phi$.
It is straightforward to check that
$$\var \mathcal S_{N, v} = \var \tilde {\mathcal S}_{N,v} \mbox{ and } \E \mathcal S_{N, v}\mathcal S_{N, u} \leq \E \tilde {\mathcal S}_{N, v}\tilde {\mathcal S}_{N,u} \mbox{ for all } u, v \in \cup_{B\in \mathcal B} B\,.$$
Combined with Lemma~\ref{lem-slepian}, it gives that
\begin{equation}\label{eq-left-tail-compare}
\P (\max_{v \in V_N} \mathcal{S}_{N,v} \leq t) \leq  \P (\max_{v \in \cup_{B\in \mathcal B} B} \mathcal{S}_{N, v} \leq t) \leq \P (\max_{v \in \cup_{B\in \mathcal B} B} \tilde {\mathcal{S}}_{N, v} \leq t) \mbox{ for all } t\in \mathbb R\,.
\end{equation}
By \eqref{eq-lefttaillb}, one has that for each $B\in \mathcal B$,
$$
\P (\sup_{v \in B} g_{N', v} \geq m_N - \lambda') 
= \P (\sup_{v \in B} g_{N', v} \geq m_{N'} + m_N - m_{N'}- \lambda') 
\geq \P (\sup_{v \in B} g_{N', v} \geq m_{N'} - \beta) \geqslant 
\frac{1}{2}\,,$$
and therefore
$$\P (\sup_{v \in \cup_{B\in \mathcal B}B} g_{N', v} < m_N - \lambda') \leq  (\tfrac12)^{|\mathcal B|}\,.$$
Thus,
$$\P(\max_{v \in \cup_{B\in \mathcal B} B} \tilde {\mathcal{S}}_{N,v} \leq m_N - \lambda) \leq \P (\sup_{v \in \cup_{B\in \mathcal B}B} g_{N', v} < m_N - \lambda') + \P(\phi \leq - \lambda') \leq C \mathrm{e}^{-c \lambda'}\,,$$
for some constants $C, c>0$. 
Combined with \eqref{eq-left-tail-compare}, this completes the proof of the lemma.
\end{proof}

\begin{proof}[Proof of Lemma~\ref{lem-GFlefttail}]
In order to prove Lemma~\ref{lem-GFlefttail}, we will compare the maximum of a  sparsified version of the log-correlated field to that of a modified version of MBRW.
By Assumption (A.1) and Lemma~\ref{lem-covapp} , there exists a $\kappa_0 = 
\kappa_0(\alpha^{(1/10)})$ such that for all $\kappa \geq \kappa_0$,
$$\V(\varphi_{2^\kappa N, 2^{\kappa}v}) \leqslant \V(\mathcal{S}_{2^{2\kappa}N, v}) \mbox{ for all } v\in V_N^{1/10}\,.$$
Therefore, one can choose a collection of positive numbers 
$\{a_v : v \in V_N^{1/10}\}$ such that 
$$\V(\varphi_{2^{\kappa}N, 2^{\kappa}v}+a_v X) = 
\V(\mathcal{S}_{2^{2\kappa}N, v}) \,,$$
where $X$ is a standard Gaussian variable. 
Since the MBRW has constant variance, we have that $|a_v - a_u | 
\leqslant C_1$ for some constant $C_1 = C_1(\alpha^{(1/10)}) > 0$. 
By Lemma~\ref{lem-covapp}  again, one has 
$$\E (\mathcal{S}_{2^{2\kappa}N, v} - 
\mathcal{S}_{2^{2\kappa}N, u} )^2  \leq  2 \log_+ |u-v| + O(1)\,,$$
where the $O(1)$ term
is bounded by a absolute constant.  
On the other hand, for all $u, v\in V_N^{1/10}$,
$$\E(\varphi_{2^{\kappa}N, 2^{\kappa}v}+a_v X -
\varphi_{2^{\kappa}N, 2^{\kappa}u}-a_u X)^2 \ge 
\log 2 \cdot \kappa + 2\log_+ |u-v| - O_{\alpha^{(1/10)}}(1)\,,$$
where $O_{\alpha^{(1/10)}}(1)$ is a term that is bounded 
by a constant depending only on $\alpha^{(1/10)}$. 
Therefore, there exists a
$\kappa = \kappa(\alpha^{(1/10)})$ such that for all $u, v\in V_N^{1/10}$,
$$\E(\varphi_{2^{\kappa}N, 2^{\kappa}v}+a_v X - \varphi_{2^{\kappa}N, 
2^{\kappa}u}-a_u X)^2 \geq \E (\mathcal{S}_{2^{2\kappa}N, v} - 
\mathcal{S}_{2^{2\kappa}N, u} )^2\,.$$
Combined with Lemma~\ref{lem-slepian}, this implies
that for a suitable $C_{\kappa}$ depending on $\kappa$,
\begin{align}\label{eq-mid-step-left-tail}
\P(\max_{v \in V_N} \varphi_{2^{\kappa}N,2^{\kappa}v}  \leqslant m_N - \lambda)& \leq \P(\max_{v \in V_N^{1/10}} (\varphi_{2^{\kappa}N, 2^{\kappa}v} +a_v X) \leqslant m_N - \lambda/2) + \P(X \leq - \lambda /C_{\kappa}) \nonumber \\
&\leq \P(\max_{v \in V_N^{1/10}} \mathcal{S}_{2^{2\kappa}N, v} \leqslant m_N - \lambda/2) +   \P(X \leq - \lambda /C_{\kappa})\,.
\end{align}
There are  number of ways to bound  
$ \P(\max_{v \in V_N^{1/10}} \mathcal{S}_{2^{2\kappa}N, v} \leqslant m_N - 
\lambda/2)$, and we choose not to optimize the bound, but instead
simply apply the FKG inequality \cite{Pitt82}. More precisely, we note that there exists a collection of boxes  $\mathcal V$ with $|\mathcal V| \leq 2^{4d\kappa}$ where each box is a translated copy of $V_N^{1/10}$ such that $V_{2^{2\kappa} N} \subseteq \cup_{V \in \mathcal V} V $. Since $\{\max_{v \in V_{2^{2\kappa}N}} \mathcal{S}_{2^{2\kappa}N, v} \leqslant m_N - \lambda/2\} = \cap_{V\in \mathcal V} \{\max_{v \in V} \mathcal{S}_{2^{2\kappa}N, v} \leqslant m_N - \lambda/2\}$, the FKG inequality gives that
$$ \P(\max_{v \in V_{2^{2\kappa}N}} \mathcal{S}_{2^{2\kappa}N, v} \leqslant m_N - \lambda/2) \geq (\P(\max_{v \in V_N^{1/10}} \mathcal{S}_{2^{2\kappa}N, v} \leqslant m_N - \lambda/2)) ^{2^{4d\kappa}}\,,$$
Combined with \eqref{eq-mid-step-left-tail} and Lemma~\ref{lem-MBRWlefttail}, 
this completes the proof of the lemma.
\end{proof}

\section{Robustness of the maximum under perturbations}
\label{sec-robust}
The main goal of this section is to establish that  the 
law of the maximum for a log-correlated Gaussian field is robust under 
certain perturbations. These invariance properties will be crucial in
Section~\ref{sec:decomposition} when
constructing a new field that approximates our target field.

For a positive integer $r$, let $\mathcal B_r$ be a collection of 
sub-boxes of side length $r$ which forms a partition of 
$V_{\lfloor N/r\rfloor r}$. Write 
$\mathcal B = \cup_{r\in [N]} \mathcal B_r$. 
Let $\{g_B: B\in \mathcal B\}$ be a collection of 
i.i.d.\ standard Gaussian variables. 
For $v\in V_N$, denote by $B_{v, r} \in \mathcal B_r$ the box that 
contains $v$. For $\sigma = (\sigma_1, \sigma_2)$ with
$\|\sigma\|_2^2=\sigma_1^2+\sigma_2^2$ 
and $r_1, r_2$, define,
\begin{equation}\label{eq-def-tilde-phi}
\tilde{\varphi}_{ N,  r_1, r_2, \sigma, v}  = 
\varphi_{N, v}+ \sigma_1 g_{B_{v, r_1}} + \sigma_2 g_{B_{v, N/r_2}} \,,
\end{equation}
and set $\tilde M_{N, r_1, r_2, \sigma} = 
\max_{v\in V_N}\tilde{\varphi}_{ N, r_1, r_2, \sigma, v}$. 

For probability measures $\nu_1,\nu_2$ on $\mathbb R$,
let
$d(\nu_1,\nu_2)$ denote the L\'{e}vy distance between $\nu_1,\nu_2$,
i.e.
$$ d(\nu_1,\nu_2)=\inf \{\delta>0: \nu_1(B)\leq \nu_2(B^\delta)+\delta\quad
\mbox{\rm for all open sets $B$}\},$$
where $B^\delta=\{y: |x-y|<\delta \mbox{  for some }x\in B\}$.
In addition, define 
$$\tilde d(\nu_1, \nu_2) = \inf \{\delta>0: \nu_1((x, \infty))\leq 
	\nu_2((x-\delta, \infty))+\delta\quad
\mbox{\rm for all } x\in \mathbb R\}\,.$$
If $\tilde d(\nu_1, \nu_2) = 0$, then $\nu_1$ is stochastically 
dominated by $\nu_2$. Thus,
$\tilde d(\nu_1, \nu_2)$  measures approximate stochastic domination of $\nu_1$ by
$\nu_2$; in particular, unlike $d(\cdot,\cdot)$, the function
$\tilde d(\cdot,\cdot)$ is not symmetric.

With a slight abuse of notation, if $X,Y$ are random variables
with laws $\mu_X,\mu_Y$ respectively, we also write
$d(X,Y)$ for $d(\mu_X,\mu_Y)$ and $\tilde d(X, Y)$ for $\tilde d(\mu_X, \mu_Y)$.

\noindent {\bf A notation convention:} By Proposition~\ref{prop-maxrtfluc}, one
has that 
$$\limsup_{\delta\to 0} \limsup_N d(\max_{v\in V_N^\delta} \varphi_{N, v}, 
\max_{v\in V_N} \varphi_{N, v}) = 0\,.$$ Therefore, in order to 
prove Theorem~\ref{thm-convergence}, it suffices to show that 
for each fixed $\delta >0$, the law of 
$\max_{v\in V_N^\delta} \varphi_{N, v}-m_N$ converges. To this end, 
one only needs to consider the Gaussian field restricted to $V_N^\delta$. 
For convenience of notation, we will treat $V_N^\delta$ as the whole box 
that is under consideration. Equivalently, throughout the rest of the 
paper when assuming (A.1), (A.2) or (A.3) holds, we assume these 
assumptions hold with $\delta = 0$, and we set 
$\alpha:=\max(\alpha_0,\alpha^{(0)}$. 

\begin{figure}[ht]
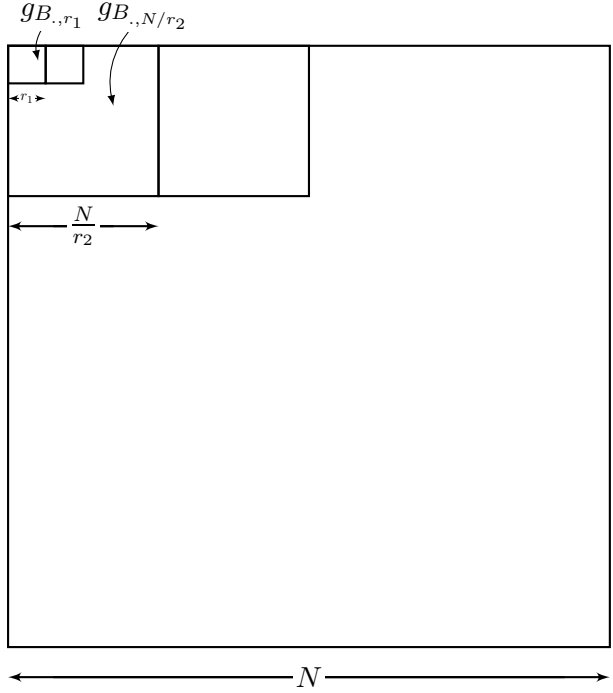

\begin{center}
\tikz[scale=2]{
\draw[black, thick](0,0) -- (0,4) -- (4,4) -- (4,0) -- cycle;
\draw[black, thick](0,4) -- (1,4) -- (1,3) -- (0,3) -- cycle;
\draw[black, thick](1,4) -- (2,4) -- (2,3) -- (1,3) -- cycle;
\draw[black, thick](0,4) -- (.25,4) -- (.25,3.75) -- (0,3.75) -- cycle;
\draw[black, thick](0.25,4) -- (.5,4) -- (.5,3.75) -- (0.25,3.75) -- cycle;
\draw[arrows=<->, -latex',  thin] (0.1, 3.65) -- (.25,3.65);
\draw[arrows=<->, -latex',  thin] (0.1, 3.65) -- (0,3.65);
\node[ann, scale=.5] at (.125, 3.65) {$r_1$};
\draw[arrows=<->, -latex',  thick] (0.3, -0.2) -- (4,-0.20);
\draw[arrows=<->, -latex',  thick] (3.7, -0.2) -- (0,-0.20);
\node[ann] at (2, -0.2) {$N$};
\draw[arrows=<->, -latex',  thick] (0.3, 2.8) -- (1,2.8);
\draw[arrows=<->, -latex',  thick] (0.7, 2.8) -- (0,2.8);
\node[ann] at (.5, 2.8) {$\frac{N}{r_2}$};
\draw[->, bend right, -latex] (0.3,4.2) to (0.2, 3.9);
\node[ann] at (.3, 4.2) {$g_{B_{.,r_1}}$};
\draw[->, bend right, -latex] (0.9,4.2) to (0.7, 3.6);
\node[ann] at (.9, 4.2) {$g_{B_{.,N/r_2}}$};
}
\caption{Perturbation levels of the  Gaussian field}
\label{fig-perturbation}
\end{center}
\end{figure}

The following lemma, which is one of the main results of this section, relates
the laws of $M_N$ and $\tilde M_{N, r_1, r_2, \sigma}$.

\begin{lemma}\label{lem-nfmax}   The following holds uniformly for all Gaussian fields $\{\varphi_{N,v}: v\in V_N\}$ satisfying Assumption (A.1):  
 \begin{equation}
 \limsup_{r_1, r_2 \to\infty} \limsup_{N\to \infty}  d(M_{N} - m_N, \tilde{M}_{N, r_1, r_2, \sigma} - m_N - \|\sigma\|_2^2\sqrt{d/2} )= 0 \,.\label{eq-invariance-2}
\end{equation} 
\end{lemma}

The next lemma states that under Assumption
(A.1), the law of the maximum is robust under
small perturbations (in the sense of $\ell_\infty$ norm) 
of the covariance matrix.

\begin{lemma}\label{lem-epsilon-domination}
Let  $\{\varphi_{N,v}: v\in V_N\}$ be a sequence of Gaussian fields satisfying 
Assumption (A.1), and let $\sigma$ be fixed. Let 
$\{\bar \varphi_{N, v}: v\in V_N\}$ be Gaussian fields such that  for all 
$u, v \in V_N$
$$|\var \varphi_{N, v} - \var \bar \varphi_{N, v}| \leq \epsilon, 
\mbox{ and } \E  \bar\varphi_{N, v} \bar\varphi_{N, u} 
\leq \E  \varphi_{N, v} \varphi_{N,  u} + \epsilon  \,.$$
Then, there exists $\iota  = \iota(\epsilon)$ with 
$\iota \to_{\epsilon\to 0} 0$ such that 
 \begin{equation*}
\limsup_{N\to \infty}  \tilde d(M_{N} - m_N , 
\max_{v\in V_N} \bar \varphi_{N, v} -  m_N ) \leq \iota \,.
\end{equation*} 
\end{lemma}

A key step in the proof of
Lemma~\ref{lem-nfmax} is the following characterization of the
geometry of
vertices achieving large values in the fields, an extension of 
\cite[Theorem 1.1]{DZ12}; it states that near maxima are either
at microscopic or macroscopic distance from each other. 
This may be of independent interest.
\begin{lemma}\label{lem-loc1}
There exists a constant  $c > 0$ such that,
uniformly for all Gaussian fields satisfying 
Assumption (A.1), 
we have
\begin{equation*}
\lim_{r \rightarrow \infty} \lim_{N \rightarrow \infty} \mathbb{P} 
(\exists u, v : |u-v| \in (r, \tfrac{N}{r}), \varphi_{N, v},\varphi_{N, u} 
\ge m_N - c \log \log r)=0\,.
\end{equation*}
\end{lemma}

\subsection{Maximal sum over restricted pairs} 
As in the case of 2D DGFF discussed in \cite{DZ12},
in order to prove Lemma~\ref{lem-loc1}, we will study the maximum of the  
sum over restricted pairs. 
For any Gaussian field $\{\eta_{N, v}: v\in V_N\}$ and $r>1$, define
\begin{align*}
\eta_{N,r}^{\diamond}=\max \{\eta_{N,u}+\eta_{N,v} : u,v 
\in V_N, r \le |u-v| \le N/r \} \,.
\end{align*}
\begin{lemma}\label{lem-expcomp1} There
	exist constants $c_1, c_2$ 
	depending only on $d$ and $C> 0$ depending only on 
	$(\alpha, d)$ such that for all $r, n$ with $N=2^n$ and all 
	Gaussian fields satisfying Assumption (A.1), we have
\begin{equation}
	\label{eq-paris3}
	2m_N - c_2 \log \log r - C \le \E\varphi_{N,r}^{\diamond} \le 2m_N - c_1 \log \log r +C\,.
\end{equation}
\end{lemma}
\begin{proof} 
In order to prove Lemma~\ref{lem-expcomp1}, we will show that
\begin{equation}\label{eq-sum-compare}
\E\mathcal{S}_{2^{-\kappa}N,r}^{\diamond}\le \E\varphi_{N,r}^{\diamond} \le \E\mathcal{S}_{2^{\kappa}N,r}^{\diamond} \,.
\end{equation}
To this end, we recall the following Sudakov-Fernique inequality 
\cite{Fernique75} which compares the first moments for maxima 
of two Gaussian processes.

\begin{lemma}\label{lem-sudfer}   Let $\mathcal{A}$ be an arbitrary 
	finite index set and let $\{ X_a : a \in \mathcal{A}\}$ and 
	$\{ Y_a : a \in \mathcal{A}\}$ be two centered Gaussian processes 
	such that: $$\E(X_a - X_b)^2 \ge \E(Y_a - Y_b)^2, 
	\qquad \mbox{ for all } a,b \in \mathcal{A}\,.$$ 
	Then $\E(\max_{a \in \mathcal{A}} X_a) \le 
	\E(\max_{a \in \mathcal{A}} Y_a)$.
\end{lemma} 
We will give a proof for the upper bound in \eqref{eq-paris3}.
The proof of the lower bound follows using similar arguments.
For $\kappa\in \mathbb N$, recall the definition of the restriction
map $\psi_N$ as in \eqref{eq-def-psi}.
By Lemma~\ref{lem-covapp}, there exists a $\kappa>0$ 
(depending only on $(\alpha, d)$) such that for all $u,v,u',v' \in V_N$, 
\begin{equation*}
\E(\varphi_{N, u}+\varphi_{N, v}-\varphi_{N, u'}-\varphi_{N, v'})^2\le \E(\mathcal{S}_{\psi_N(u)}^{2^{\kappa}N}+\mathcal{S}_{\psi_N(v)}^{2^{\kappa}N}-\mathcal{S}_{\psi_N(u')}^{2^{\kappa}N}-\mathcal{S}_{\psi_N(v')}^{2^{\kappa}N})^2\,.
\end{equation*}
(To see this, note that the variance of $ \mathcal{S}_{\psi_N(u)}^{2^{\kappa}N}$ 
increases with $\kappa$ but the covariance between 
$\mathcal{S}_{\psi_N(u)}^{2^{\kappa}N}$ and 
$\mathcal{S}_{\psi_N(v)}^{2^{\kappa}N}$ does not.)
In addition, note that for $r \le |u-v| \le N/r$ one has
$r \le |\psi_N(u)-\psi_N(v)|\le 2^{\kappa} N/r$. 
Combined with Lemma \ref{lem-sudfer},  this yields
$\E{\varphi}_{N,r}^{\diamond} \le \E\mathcal{S}_{2^{\kappa}N,r}^{\diamond}$,
completing the proof of the upper bound in \eqref{eq-sum-compare}.

To complete the proof of Lemma
\ref{lem-sudfer},
note that \cite[Lemma 3.1]{DZ12} readily extends 
to MBRW in $d$-dimension, and thus
$$2m_N - c_2 \log \log r - C \le 
\E\mathcal S_{N,r}^{\diamond} \le 2m_N - c_1 \log \log r +C\,,$$
where $c_1, c_2$ are constants depending only on $d$ and $C$ is a constant depending on $(\alpha, d)$.
Combined with \eqref{eq-sum-compare}, this completes the proof of the
lemma.
\end{proof}
We will also need the following tightness result. 
\begin{lemma}\label{lem-tight} 
Under Assumption (A.1), the sequence
 $\{ (\varphi^{\diamond}_{N , r} - 
\E \varphi^{\diamond}_{N , r})/ \log \log r \}_{N\in \mathbb N, r\geq 100 }$ 
is tight. Further, there exists a constant $C>0$ depending only on 
$d$ such that for all $r\geq 100$ and $N\in \mathbb N$,
$$|(\varphi^{\diamond}_{N , r} - 
\E \varphi^{\diamond}_{N , r})| \leq C \log \log r\,.$$
\end{lemma}
\begin{proof} 
Take $N' =  2N$ and partition $V_{N'}$ into $2^d$ copies of 
$V_N$, denoted by $V_N^{(1)} , \ldots, V_N^{(2^d)}$. 
For each $i \in [2^d]$, let $\{\varphi_{N, v}^{(i)}: v\in V_N^{(i)}\}$ be
an independent copy of $\{\varphi_{N, v}: v\in V_n\}$ where we 
identify $V_N$ and $V_N^{(i)}$ by the suitable translation 
such that the two boxes coincide. Denote by 
\begin{equation}\label{eq-def-tilde-phi1}
\hat \varphi_{N', v} = \varphi_{N, v}^{(i)} 
\mbox{ for } v\in V_N^{(i)} \mbox{ and } i\in [2^d]\,.
\end{equation}
Clearly, $\{\varphi_{N', v}\}$ is a Gaussian field that 
satisfies Assumption (A.1) (with $\alpha$ increased by an 
absolute constant). Therefore, by Lemma~\ref{lem-expcomp1}, we have 
\begin{equation}\label{eq-exp-sum-tilde}
2m_N - c_2 \log \log r - C \le  \E 
\hat\varphi_{N,r}^{\diamond} \le 2m_N - c_1 \log \log r +C\,,
\end{equation}
where $c_1, c_2, C>0$ are constants depending only on $(d, \alpha)$.
In addition, we have $$\E (\hat{\varphi}_{N',r}^{\diamond}) 
\geq \E \max \{\varphi_{N,r}^{(1), \diamond},\varphi_{N,r}^{(2), 
\diamond} \}\,.$$
Combined with Lemma~\ref{lem-expcomp1} and 
\eqref{eq-exp-sum-tilde}, and the simple algebraic 
fact that $|a-b| = 2 (a \vee b) -a - b$, it yields that
\begin{equation*} \E|\varphi_{N,r}^{(1), \diamond} - \varphi_{N,r}^{(2), 
	\diamond} | \le 2(\E \hat{\varphi}_{N',r}^{\diamond}-
	\E \varphi_{N,r}^{\diamond} ) \le C' \log
	\log r\,, \mbox{ for all } r\geq 100\,,
\end{equation*}
where $C'>0$ is a constant depending only on $d$. 
This completes the proof of the lemma.
\end{proof}

\subsection{Proof of Lemma~\ref{lem-loc1}}
In this subsection we will prove Lemma~\ref{lem-loc1}, by contradiction. 
Suppose otherwise that Lemma~\ref{lem-loc1} does not hold. 
Then for any constant $c>0$, there exists $\epsilon > 0$ and a 
subsequence $\{r_k\}$ such that for all $k\in \mathbb N$
\begin{equation}\label{eq-loc2}\lim_{N \rightarrow \infty} 
	\P \big( \exists u, v : |u-v| \in \left(r_k, \frac{N}{r_k}\right), 
	\varphi_{N, v},\varphi_{N, u} \ge m_N - c \log \log r_k\big)> \epsilon\,.
\end{equation}
Now fix $\delta>0$ and consider $N' = 2^\kappa N$ where 
$\kappa$ is an integer to be selected.  
Partition $V_{N'}$ into $2^{\kappa d}$ disjoint boxes of side 
length $N$, denoted by $V_N^{(1)}, \ldots, V_N^{(2^{\kappa d})}$. 
Define $\{\hat \varphi_{N', v}: v\in V_{N'}\}$ in the same manner 
as in \eqref{eq-def-tilde-phi1} except that now we take 
$2^{\kappa d}$ copies of $\{\varphi_{N, v}: v\in V_N\}$ 
(one for each $V_{N}^{(i)}$ with $i\in [2^{\kappa d}]$). Clearly, 
$\{\hat \varphi_{N', v}: v\in V_{N'}\}$ is a Gaussian field 
satisfies Assumption (A.1) with $\alpha$ replaced by 
a constant $\alpha'$ depending only on $(\alpha, d, \kappa)$. 
Therefore, by Lemma~\ref{lem-expcomp1}, 
\begin{equation}\label{eq-exp-sum-tilde-kappa}
2m_N - c_2 \log \log r - C \le  \E \hat\varphi_{N',r}^{\diamond} \le 2m_N - c_1 \log \log r + C\,,
\end{equation}
where $c_1, c_2>0$ are two constants depending only on $d$ and $C>0$ is a constant depending only on $(\alpha, d, \kappa)$.

Next we derive a contradiction to \eqref{eq-exp-sum-tilde-kappa}. 
Set $z_{N, r} = 2m_N - c\log \log r$, $Z_{N, r} = 
(\hat{\varphi}_{N',r}^{\diamond} - z_{N, r})_-$ 
and $Y_{N, r}^{(i)} = (\varphi_{N,r_k}^{(i), \diamond}  - 
z_{N, r})_-$. Then \eqref{eq-loc2} implies that 
\begin{equation}\label{eq-Y-N-r}
\lim_{N\to \infty}\P(Y_{N, r_k}^{(1)} > 0) \leq 1 - \epsilon\, \mbox{ for all } k\in \mathbb N\,.
\end{equation}
In addition, by Lemmas~\ref{lem-expcomp1} and \ref{lem-tight}, there exists a constant $C'>0$ depending only on $d$ such that for all $r\geq 100$ and $N\in \mathbb N$, we have
\begin{equation}\label{eq-Y-N-r-exp}
\E Y_{N, r}^{(1)} \leq C' \log\log r\,.
\end{equation}
Clearly,  $Z_{N, r} \leq \min_{i\in [2^{\kappa d}]} Y_{N, r}^{(i)}$. Combined with the fact that $Y_{N, r}^{(i)}$ are i.i.d.\ random variables, one obtains
$$\E Z_{N, r_k}\leq \int_0^\infty (\P(Y_{N, r_k}^{(1)} > y))^{2^{\kappa d}} dy \leq (1-\epsilon)^{2^{\kappa d } - 1} \int_0^\infty (\P(Y_{N, r_k}^{(1)} > y))dy \leq  (1-\epsilon)^{2^{\kappa d } - 1} \E Y_{N, r_k}^{(1)}\,,$$
where \eqref{eq-Y-N-r} was used in the second inequality. 
Combined with \eqref{eq-Y-N-r-exp}, one concludes that for all $r\geq 100$ 
and $N$
$$\E Z_{N, r_k} \leq  (1-\epsilon)^{2^{\kappa d } - 1} C' \log\log r_k\,.$$
Now set $c = c_1/4$ and choose $\kappa$ depending on 
$(\epsilon, d, C', c_1)$ such that 
$ (1-\epsilon)^{2^{\kappa d } - 1} C' \leq c_1/4$. Then,
$$\E \hat \varphi_{N', r_k}^\diamond \geq 2m_N - c_1 \log \log r_k/2\,,$$
for all $k\in \mathbb{N}$ and sufficiently large 
$N \geq N_{ k}$ where $N_{ k}$ is a number depending only on $k$. 
Sending $N\to \infty$ first and then $k\to \infty$ 
contradicts \eqref{eq-exp-sum-tilde-kappa}, 
thereby completing the proof of the lemma. \qed

\subsection{Proof of Lemmas~\ref{lem-nfmax} and \ref{lem-epsilon-domination}}
The next lemma, which 
extends \cite[Lemma 3.9]{BDZ13} to the current setup,
will be useful for the proof of Lemma~\ref{lem-nfmax}  and later in the paper. 
\begin{lemma}\label{lem-robmax}
Let Assumptions (A.0) and (A.1) holds. 
Let $\{\phi_u^N: u\in V_N\}$ be a collection of random variables 
independent of $\{\varphi_{N,  u}: u\in V_N\}$ such that
\begin{equation}\label{eq-assumption-phi}
\P(\phi_u^N \ge 1 + y) \le e^{-y^2} \mbox{ for all } u\in V_N\,.
\end{equation}
Then, there exists  $C = C(\alpha, d)> 0$ such that, 
for any $\epsilon > 0, N \in \mathbb{N} $  and $x \ge - \epsilon ^{-1/2}$,
\begin{equation}\label{eq-robrttailgf}
\P( \max_{u \in V_N} ( \varphi_{N, u} + \epsilon \phi_{u}^N ) 
\ge m_{N} + x) \le \P( \max_{u \in V_N}  \varphi_{N, u} \ge m_{N} + x - \sqrt{\epsilon} ) (1 + C(e^{-C^{-1}\epsilon^{-1}})) \,.
\end{equation}
\end{lemma}
\begin{proof}
We first 
give the proof for $\epsilon \le 1$.
Define $ \Gamma_y =\{ u \in V_N : y/2 \le \epsilon \phi_u^N \le y \}$.  Then,
\begin{align*}
\P ( \max_{u \in V_N} (\varphi_{N, u} + \epsilon \phi_u^N ) \ge m_{N} + x) 
\le &\P( M_{N}\ge  m_{N} + x - \sqrt{\epsilon} ) \\
& + \sum_{i=0}^{\infty} \E ( \P (\max_{u \in \Gamma_{2^i \sqrt{\epsilon}}} \varphi_{N, u} \ge  m_{N} + x - 2^i\sqrt{\epsilon}| \Gamma_{2^i \sqrt{\epsilon}}))\,.
\end{align*}
By Proposition \ref{prop-maxrtfluc}, one
can bound the second term on the right hand side above by
\begin{equation*}
\sum_{i=0}^{\infty} \E ( \P (\max_{u \in V_N}   \varphi_{N,  u} \ge  m_{N} + 
x - 2^i\sqrt{\epsilon}| \Gamma_{2^i \sqrt{\epsilon}})) \lesssim 
\frac{x \vee 1}{e^{\sqrt{2d}x}} \sum_{i=0}^{\infty} 
\E ( |\Gamma_{2^i \sqrt{\epsilon} } | /N^d)^{1/2} 
e^{\sqrt{2 d}2^i \sqrt{\epsilon} } \,.
\end{equation*}
By \eqref{eq-assumption-phi}, one has
$\E(|\Gamma_{2^i \sqrt{\epsilon} } | /N^d)^{1/2} \le 
e^{-4^i (C\epsilon)^{-1}} $. Altogether, one gets
\begin{equation*}
\sum_{i=0}^{\infty} \E ( \P (\max_{u \in V_N} \varphi_{N, u} 
\ge m_{N} + x - 2^i\sqrt{\epsilon}| \Gamma_{2^i \sqrt{\epsilon}})) 
\lesssim \frac{x \vee 1}{e^{\sqrt{2d}x}} e^{- (C\epsilon)^{-1}} \,,
\end{equation*}
completing the proof of the lemma when $\epsilon\leq 1$. The case $\epsilon> 1$
is simpler and follows by repeating the same argument with $\Gamma_{2^i\epsilon}$
replacing $\Gamma_{2^i\sqrt{\epsilon}}$. We omit further details.
\end{proof}

We next consider a combination of two independent copies of $\{\varphi_{N,v}\}$.
For $\sigma>0$, define 
\begin{equation}\label{eq-def-phi*}
\varphi^*_{N, \sigma, v} =\varphi_{N, v} + 
\sqrt{\frac{\|\sigma\|_2^2}{\log N}} \varphi'_{N, v} \mbox{ for } 
v\in V_N\,, \mbox{ and } M^*_{N, \sigma} = 
\max_{v\in V_N} \varphi^*_{N, \sigma, v}\,.  
\end{equation} 
where $\{\varphi'_{N, v}: v\in V_N\}$ is an 
independent copy of $\{\varphi_{N, v}: v\in V_N\} $. 
Note that the field $\{\varphi^*_{N,\sigma,v}\}$ is distributed like  the field
$\{a_N\varphi_{N,v}\}$ where $a_N=\sqrt{1+\|\sigma\|_2^2/\log N}$.
\begin{remark}\label{remark-BL}
The idea of writing a Gaussian field as a sum of two independent Gaussian 
fields has been extensively employed in the study of Gaussian processes. 
In the context of the study of extrema of the 2D DGFF,
this idea was first used in \cite{BL13}, where (combined with an invariance
result from 
\cite{Liggett} as well as the geometry of the maxima of DGFF \cite{DZ12}, see
Lemma
\ref{lem-expcomp1}) it led to a complete description of the extremal
process of 2D DGFF. 
The definition \eqref{eq-def-phi*} is inspired by \cite{BL13}.
\end{remark}
The following is the key to the proof of Lemma~\ref{lem-nfmax}.
\begin{prop}\label{prop-equivfields} 
Let Assumption (A.1) hold.
Let $\{\tilde \varphi_{N, r, \sigma,v}: v\in V_N\}$ and 
$\{\varphi^*_{N, \sigma, v}: v\in V_N\}$ be defined as in 
\eqref{eq-def-tilde-phi} and \eqref{eq-def-phi*} respectively. 
Then for any fixed $\sigma$, 
\begin{equation}
	\label{eq-paris4}
\lim_{r_1, r_2\to \infty}\limsup_{N\to \infty}d(\tilde M_{N, r_1, r_2, 
\sigma} - m_N, M^*_{N, \sigma} - m_N) = 0\,.
\end{equation}
\end{prop} 
\begin{proof}
Partition $V_N$ into boxes of side length $N/r_2$ and denote by 
$\mathcal B$ the collection of boxes. Fix an arbitrary small 
$\delta>0$, and let $B_\delta$ denote the box in the center of 
$B$ with side length $(1-\delta)N/r_2$ for each $B\in \mathcal B$. 
Write $V_{N, \delta} = \cup_{B\in \mathcal B} B_\delta$. 
Set $\tilde M_{N, r_1, r_2, \sigma,  \delta} = 
\max_{v\in V_{N, \delta}} \tilde \varphi_{N, r_1, r_2, \sigma, v}$ 
and $M^*_{N, \sigma,  \delta} = \max_{v\in V_{N, \delta}} 
\varphi^*_{N, \sigma, v}$. 
By \eqref{eq-maxrtflucrttail2}, one has
$$\lim_{\delta \to 0} \lim_{N\to \infty}\P(\tilde M_{N, r_1, 
r_2 ,\sigma, \delta} \neq \tilde M_{N, r_1, r_2, \sigma}) = 
\lim_{\delta \to 0} \lim_{N\to \infty} \P(M^*_{N, \sigma,  
\delta} \neq M^*_{N, \sigma}) = 0\,.$$
Therefore, it suffices to prove
\eqref{eq-paris4} with
$\tilde M_{N, r_1, r_2, \sigma, \delta}$ and 
$M^*_{N, \sigma,  \delta}$ replacing 
$\tilde M_{N, r_1, r_2, \sigma}$ and 
$M^*_{N, \sigma}$.
To this end, let $z_{B}$ be such that 
$$ \max_{v \in B_\delta } \varphi_{N, v} = 
\varphi_{N, z_B} \mbox{ for every } B\in \mathcal B\,.$$
We will show below that 
\begin{align}\label{eq-fine-field}
&\lim_{r_1, r_2\to \infty} \limsup_{N\to \infty} 
\P(|\tilde M_{N,r_1, r_2, \sigma, \delta} - \max_{B\in \mathcal B} 
\tilde\varphi_{N, r_1, r_2, \sigma, z_B} | \geq 1/\log\log N) \nonumber \\
=& \limsup_{N\to \infty} \P( |M^*_{N, \sigma, \delta} - \max_{B\in \mathcal B} \varphi^*_{N, \sigma, z_B} | \geq 1/\log\log N) = 0\,.
\end{align}
 Note that conditioning on the field $\{\varphi_{N, v}: v\in V_N\}$, 
 the field $\{ \sqrt{{\|\sigma\|_2^2}/{\log N}} \varphi'_{N, v}: 
 B\in \mathcal B\}$ is centered Gaussian field with 
 pairwise correlation bounded by $O(1/\log N)$. Therefore
 the conditional covariance matrix of $\{ \sqrt{\frac{\|\sigma\|_2^2}{\log 
 N}}\varphi'_{N, z_B}: B\in \mathcal B\}$ and that of
 $\{\sigma_1 g_{B_{z_B, r_1}} + \sigma_2 g_{B_{z_B, N/r_2}}: 
 B\in \mathcal B\}$ are within additive $O(1/\log N)$ of each other 
 entrywise. Combined with \eqref{eq-fine-field}, it then yields the proposition.

It remains to prove
\eqref{eq-fine-field}.
Write $r = r_1 \wedge r_2$ and
let $C$ be a constant
which we will send to infinity after sending 
first $N\to \infty$ and then $r\to\infty$,
and let $c$ be the constant from Lemma~\ref{lem-loc1}. 
Suppose that either of  the events 
that are considered in \eqref{eq-fine-field} occurs. In this case, 
one of the following events has to occur:
\begin{itemize}
\item The event $E_1 = \{\tilde{M}_{N, r_1, r_2, \sigma,  \delta} 
\not\in (m_N - C, m_N + C)\} \cup \{M^*_{N, \sigma,  \delta} 
\not\in (m_N - C, m_N + C)\}$.
\item The event $E_2$ that there exists $u, v \in (r,N/r)$ 
	such that $\varphi_{N, u} \wedge \varphi_{N, v} > m_N - c\log \log r$.
\item The event $E_3 = \tilde E_3 \cup E^*_3$ where $\tilde E_3$ ($E^*_3$) 
	is the event that $\tilde{M}_{N, r_1, r_2, \sigma} $ 
	($M^*_{N, \sigma,  \delta}$) is achieved at a vertex
	$v$ such that $\varphi_{N, v} \le m_N - c\log \log r$.
\item The event $E_4$ that there exists $v\in B\in \mathcal B$ with $\varphi_{N, v} \geq m_N - c\log\log r$ and 
$$ \sqrt{\tfrac{\|\sigma\|_2^2}{\log N}} \varphi'_{N, v}  -  \sqrt{\tfrac{\|\sigma\|_2^2}{\log N}} \varphi'_{N, z_B} \geq \tfrac{1}{\log\log N}\,.$$
\end{itemize} 
By Theorem \ref{thm-tightness}, 
$\lim_{C\to \infty} \limsup_{N\to \infty} \P(E_1) = 0$. 
By Lemma \ref{lem-loc1}, 
$\lim_{r\to \infty} \limsup_{N\to \infty}\P(E_2) = 0$. 
In addition, writting $\Gamma_x =\{ v \in V_N : \tilde{\varphi}_{N, r_1, r_2, 
\sigma, v} - \varphi_{N, v} \in (x, x+1) \}$, one has
\begin{align*}
\P(E_1^c \cap \tilde E_3) & \leq 
\P( \max_{x \ge c\log\log r - C} \max_{ v \in \Gamma_x } 
\tilde{\varphi}_{N, r_1, r_2, \sigma, v} \ge m_N - C )\\ 
& \le 
\sum_{x \ge c\log\log r - C} \P(  \max_{ v \in \Gamma_x } 
\tilde{\varphi}_{N, r_1, r_2, \sigma, v} \ge m_N-C ) \\
& \le  \sum_{x \ge c\log \log r -C}\E( \P( \max_{ v \in \Gamma_x }
{\varphi}_{N, v} \ge m_N -x - C   | \Gamma_x )) \\
& 
\lesssim_C 
\sum_{x \ge c\log\log r-C}\E( | \Gamma_x |/N^d )^{1/2}x e^{\sqrt{2 d} x} \,,
\end{align*}
where the last inequality follows from \eqref{eq-maxrtflucrttail2}.
From simple estimates using the Gaussian distribution
one has $\E (| \Gamma_x|/N^d)^{1/2} \leq e^{-c' x^2}/c'$ where 
$c' = c'(\sigma) >0$. Therefore, one concludes that
$$\limsup_{C\to \infty}\limsup_{r\to \infty} 
\limsup_{N\to \infty}\P(E_1^c\cap \tilde E_3) = 0\,.$$ 
A similar argument leads to the same estimate with $E^*_3$ replacing
$E_3$. Thus,
$$\limsup_{C\to \infty}\limsup_{r\to \infty} \limsup_{N\to \infty}\P(E_1^c\cap E_3) = 0\,.$$

Finally, let $\Gamma'_r = \{v: \varphi_{N, v} \geq m_N - c\log\log r\}$.
On the event $E_2^c$, one has  $|\Gamma'_r| \leq r^4$. 
Further, for each $v\in B\cap \Gamma'_r$, on $E_2^c$ one
has 
$|v - z_B| \leq r$ and thus (by the
independence between $\{\varphi_{N, v}\}$ and $\{\varphi'_{N, v}\}$),
$$\P( \sqrt{\tfrac{\|\sigma\|_2^2}{\log N}} \varphi'_{N, v}  -  \sqrt{\tfrac{\|\sigma\|_2^2}{\log N}} \varphi'_{N, z_B} \geq 1/\log\log N) = o_N(1)\,.$$
Therefore, a union bound gives that 
$$\limsup_{r\to \infty}\limsup_{N\to \infty}\P(E_4\cap E_2^c) \leq \limsup_{r\to \infty}\limsup_{N\to \infty} r^4 o_N(1) =0\,.$$
Altogether, this completes the proof of 
\eqref{eq-fine-field} and hence 
of the proposition.
\end{proof}

\begin{proof} [Proof of Lemma~\ref{lem-nfmax}]
Define \begin{equation*}
                  \bar{\varphi}_{N, \sigma, v} = \left(1 + \frac{ \|\sigma\|_2^2 }{2\log N}\right) \varphi_{N, v} \mbox{ for } v\in V_N\,, \mbox{ and } \bar{M}_{N, \sigma}=  \max_{v \in V_N} \bar{\varphi}_{N, v}\,.
                    \end{equation*}
Clearly we have $\bar M_{N, \sigma}  = (1+ \frac{\|\sigma\|_2^2}{2\log N})M_N $. Combined with \eqref{eq-thm-expectation}, it gives that $\E \bar M_{N, \sigma} = \E M_N + \sigma^2 \sqrt{d/2} + o(1)$ and that $d(M_N - \E M_N,\bar{M}_{N, \sigma} - \E \bar{M}_{N, \sigma} ) \rightarrow 0$ as $N\to \infty$. 
Further define $\{\varphi^*_{N, \sigma, v}: v\in V_N\}$ as in \eqref{eq-def-phi*}.
By the fact that the field $\{\bar \varphi_{N, \sigma, v}\}$ can be seen as a sum of $\{\varphi^*_{N, \sigma,v}\}$ and an independent field whose variances are $O((1/\log N)^3)$ across the field, we see that $\E \bar M_{N, \sigma} = \E M^*_{N, \sigma} + o(1)$ and that
\begin{equation}\label{eq-M-N-M*-N}
d(\bar M_{N, \sigma} - \E \bar M_N,M^*_{N, \sigma} -  \E M^*_N ) \rightarrow 0\,.
\end{equation}
Combined with Proposition~\ref{prop-equivfields}, this completes the proof of the lemma.
\end{proof}

\begin{proof}[Proof of Lemma~\ref{lem-epsilon-domination}] 
	Let $\phi$ and $\phi_{N, v}$ be i.i.d.\ standard Gaussian variables, 
	and for $\epsilon^*>0$ let 
$$\varphi_{\mathrm{lw}, N, \epsilon^*, v} = (1 - \epsilon^*/\log N)  
\varphi_{N,v}+ \epsilon'_{N, v} \phi \mbox { and } 
\bar \varphi_{\mathrm{up}, N, \epsilon^*, v} = 
(1 - \epsilon^*/\log N) \bar \varphi_{N, v} + \epsilon''_{N, v} \phi_{N, v}\,,$$
where $\epsilon'_{N, v}, \epsilon''_{N, v}$ are chosen so that 
$\var \varphi_{\mathrm{lw}, N , \epsilon^*, v} = \var \bar \varphi_{\mathrm{up}, N, 
\epsilon^*, v}  = \var \varphi_{N, v} + \epsilon$. We can choose 
$\epsilon^* = \epsilon^*(\epsilon)$ with $\epsilon^*\to_{\epsilon\to 0} 0$ so that 
$\E \varphi_{\mathrm{lw}, N, \epsilon^*, v} \varphi_{\mathrm{lw}, N, \epsilon^*, u}
\geq \E  \bar \varphi_{\mathrm{up}, N, \epsilon^*, v}  
\bar \varphi_{\mathrm{up}, N, \epsilon^*, u} $ for all $u, v\in V_N$. 
By Lemma~\ref{lem-slepian}, one has
$$\tilde d (\max_{v\in V_N} \varphi_{\mathrm{lw}, N, \epsilon^*, v} -  m_N, 
\max_{v\in V_N} \bar \varphi_{\mathrm{up}, N, \epsilon^*, v} -  m_N ) = 0\,.$$
Combined with Lemma~\ref{lem-robmax}, this completes the proof of the lemma.
\end{proof}

\section{Proofs of  
Theorems~\ref{thm-convergence} and \ref{explicit-limit-law}}  \label{sec:approx}

In this section we assume (A.0)--(A.3) and
prove Theorem~\ref{thm-convergence}. Toward
this end, in Subsection~\ref{sec:decomposition} we will approximate the field 
$\{\varphi_{N, v}: v\in V_N\}$ by a simpler to analyze field, in such a way that
the results of Section \ref{sec-robust} apply and yield the asymptotic
equivalence of their respective laws of the centered maximum.
In Subsection~\ref{sec:limit-law} we prove the convergence in law for the centered maximum of the new field. Our method of proof yields 
Theorem~\ref{explicit-limit-law} as a byproduct.

\subsection{An approximation of the log-correlated Gaussian field} 
\label{sec:decomposition}

In this subsection, we approximate the log-correlated Gaussian field. 
Let $R_N(u,v)=E(\varphi_{N,u}\varphi_{N,v})$.
We consider three scales for the approximation of the field
$\{\varphi_{N,v}\}$:
\begin{enumerate}
	\item
The top (macroscopic) scale, dealing with $R_N(u,v)$ for
$|u-v| \asymp N$.
\item  The bottom (microscopic) 
	scale, dealing with $R_N(u,v)$  for
	$|u-v| \asymp 1$.
\item  The middle (mesoscopic) scale, dealing with $R_N(u,v)$ for
$1\ll |u-v| \ll N$. 
\end{enumerate}
By Assumptions (A.2) and (A.3), $R_N(u,v)$, properly centered, 
converges in the top and bottom scale. So in those scales,
we approximate $\{ \varphi_{N, u}\}$ by the
corresponding ``limiting'' fields. 
In the middle scale, we simply approximate $\{\varphi_{N, u}\}$ by the MBRW. 
One then expects that this approximation gives an additive $o(1)$ 
error for $R_N(u,v)$  in the top and bottom scale, and an 
additive $O(1)$ error in the middle scale. 
It turns out that this guarantees that the limiting laws 
of the centered maxima coincide.

In what follows, for any integer $t$ we refer to
a box of side length $t$ as an $t$-box.
Take two large integers $L = 2^\ell$ and $K=2^k$. Consider first
$\{\varphi_{KL, u}: u\in V_{KL}\}$ in a $KL$-box whose left-bottom corner 
is identified as the origin, and let 
$\Sigma$ denote its covariance matrix.

Recall that
by Proposition~\ref{prop-maxrtfluc}, with probability tending to 1 as 
$N\to \infty$, the maximum of $\varphi_{N,v}$ 
over $V_N$ occurs in a sub-box of $V_N$ with side length 
$\lfloor N/KL \rfloor \cdot KL$. 
Therefore, one may neglect the maximization over the indices in $V_N\setminus 
V_{\lfloor N/KL \rfloor \cdot KL}$. For notational
convenience, we will 
assume  throughout that $KL$ divides $N$ in what follows. 

We use $\Sigma$ to approximate the macroscopic scale of $R_N(u,v)$, 
as follows.
Partition $V_N$ into a disjoint union of boxes of side length 
$N/KL$, denoted
$\mathcal B_{N/KL} = \{B_{N/KL, i}: i=1, \ldots, (KL)^d\}$. 
Let $v_{N/KL, i}$ be the left bottom corner of box 
$B_{N/KL, i}$ and write $w_i = \frac{v_{N/KL, i}}{N/KL}$. 
Let $\Xi^c$ be a matrix of dimension $N^d \times N^d$ such 
that $\Xi^c_{u, v} = \Sigma_{w_i, w_j}$ for $u\in B_{N/KL, i}$ and 
$v\in B_{N/KL, j}$. Note that $\Xi^c$ is a positive 
definite matrix with diagonal terms $\log (KL) + O_{KL}(1)$. 

Next, take  two other 
integers $K' = 2^{k'}$ and $L' = 2^{\ell'}$. As above, we assume that 
$K'L'$ divides $N$. Consider 
$\{\varphi_{K'L', u}: u\in V_{K'L'}\}$ in a $K'L'$-box 
whose left-bottom corner is identified as the origin, and
denote by $\Sigma'$ the covariance matrix for 
$\{\varphi_{K'L', u}: u\in V_{K'L'}\}$. As above, 
assume for notational convenience that $K'L'$ divides $N$.
Partition $V_N$ into a 
disjoint union of boxes of side length $K'L'$, 
denoted 
$\mathcal B_{K'L'} = \{B_{K'L', i}: i=1, \ldots, (N/K'L')^d\}$.
Let $v_{K'L', i}$ be the left bottom corner of $B_{K'L', i}$. 
Let $\Xi^b$ be a matrix of dimension 
$N^d \times N^d$ so that 
$$\Xi^b_{u, v} = \left\{
	\begin{array}{ll}
		\Sigma'_{u - v_{K'L', i}, v - v_{K'L', i}}, &  
u, v \in B_{K'L', i}\\
0, & 
u \in B_{K'L', i}, v\in B_{K'L', j}, i\neq j
\end{array}
\right. . 
$$
Note that $\Xi^b$ is a positive definite matrix with diagonal terms 
$\log (K'L') + O_{K'L'}(1)$.

Let  $\{ \xi_{N,v}^c : v \in V_N \}$ be a Gaussian field with 
covariance matrix $\Xi^c$, which we occasionally refer to as the coarse 
field, and let $\{ \xi^b_{N, v} : v \in V_N \}$ be a Gaussian field with 
covariance matrix $\Xi^b$, which we occasionally refer to as the bottom field. 
Note that the coarse field is constant in each box $B_{N/KL, i}$, 
and the bottom fields in different boxes $B_{K'L', i}$ are independent of 
each other.

We will consider the limits when $L,K, L', K'$ are sent to infinity in 
that order. In what follows, we denote by $(L,K,L',K') \Rightarrow 
\infty$ sending these parameters to infinity in the order of $K', L', K, L$ 
(so $K'\gg L' \gg K \gg L$).

\begin{figure}[ht]
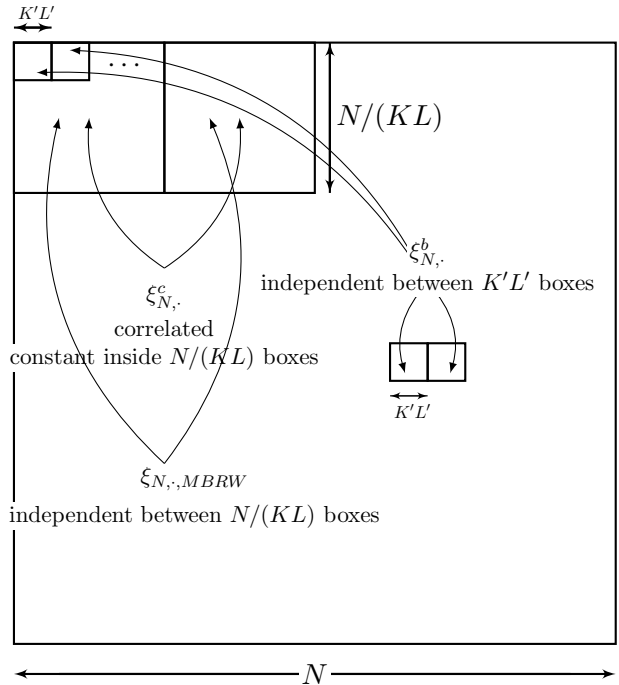

	\begin{center}
		\tikz[scale=2]{
			\draw[black, thick](0,0) -- (0,4) -- (4,4) -- (4,0) -- cycle;
			\draw[black, thick](0,4) -- (1,4) -- (1,3) -- (0,3) -- cycle;
			\draw[black, thick](1,4) -- (2,4) -- (2,3) -- (1,3) -- cycle;
			\draw[black, thick](0,4) -- (0.25,4) -- (0.25,3.75) -- (0,3.75) -- cycle;
			\draw[black, thick](0.25,4) -- (0.5,4) -- (0.5,3.75) -- (0.25,3.75) -- cycle;
			\node[ann] at (0.75, 3.85) {$\cdots$};
			\node[ann] at (2.5,  3.5) {$\cdots$};
			\draw[black, thick](2.5,2) -- (2.75,2) -- (2.75,1.75) -- (2.5,1.75) -- cycle;
			\draw[black, thick](2.75,2) -- (3,2) -- (3,1.75) -- (2.75,1.75) -- cycle;
			\draw[arrows=<->, -latex',  thin] (2.5, 1.65) -- (2.75,1.65);
			\draw[arrows=<->, -latex',  thin] (2.75, 1.65) -- (2.5,1.65);
			\node[ann, scale=.6] at (2.65, 1.55) {$K'L'$};
			\draw[arrows=<->, -latex',  thick] (0.3, -0.2) -- (4,-0.20);
			\draw[arrows=<->, -latex',  thick] (3.7, -0.2) -- (0,-0.20);
			\node[ann] at (2, -0.2) {$N$};
			\draw[arrows=<->, -latex',  thick] (2.1, 4) -- (2.1,3);
			\draw[arrows=<->, -latex',  thick] (2.1, 3) -- (2.1,4);
			\draw[arrows=<->, -latex',  thick] (0, 4.1) -- (0.25,4.1);
			\draw[arrows=<->, -latex',  thick] (0.25, 4.1) -- (0,4.1);
			\node[ann, scale=.6] at (0.15,4.2) {$K'L'$};
			\node[ann] at (2.5,3.5) {$N/(KL)$};
			\draw[->, bend right, -latex] (2.75,2.4) to (2.6, 1.8);
			\draw[->, bend left, -latex] (2.75,2.4) to (2.9, 1.8);
			\draw[->, bend right, -latex] (2.75,2.4) to (0.15,3.80);
			\draw[->, bend right, -latex] (2.75,2.4) to (0.37,3.95);
			\node[ann, scale=.8] at (2.75, 2.6) {$\xi^b_{N,\cdot}$};
			\node[ann, scale=.8] at (2.75, 2.4) {independent between $K'L'$ boxes};
			\draw[->, bend left, -latex] (1,2.5) to (0.5, 3.5);
			\draw[->, bend right, -latex] (1,2.5) to (1.5, 3.5);
			\node[ann, scale=.8] at (1, 2.3) {$\xi^c_{N,\cdot}$};
			\node[ann, scale=.8] at (1, 2.1) {correlated};
			\node[ann, scale=.8] at (1, 1.9) {constant inside $N/(KL)$ boxes};
			\draw[arrows=<->, bend left, -latex] (1,1.2) to (0.3, 3.5);
			\draw[arrows=<->, bend right, -latex] (1,1.2) to (1.3, 3.5);
			\node[ann, scale=.8] at (1.2, 1.1) {$\xi_{N,\cdot,MBRW}$};
			\node[ann, scale=.8] at (1.2, 0.85) {independent between $N/(KL)$ 
			boxes};
		}
		\caption{Hierarchy of construction of the approximating Gaussian field}
		\label{fig-perturbation1}
	\end{center}
\end{figure}

Finally,
we give the MBRW approximation for the mesoscopic scales. 
Recall the definitions of $\mathcal B_j^N$ and $\mathcal B_j(v)$ 
in Subsection~\ref{sec:MBRW}, and recall that $\{ b_{i, k,B} : k \ge 0, 1\leq i\leq (KL)^d, 
B \in \mathcal{B}_k^N\}$ is a family of independent Gaussian 
variables such 
that $ \var b_{i, j,  B} = \log 2\cdot  2^{-dj}$ for all 
$B\in \mathcal B_j^N$ and $1\leq i\leq (KL)^d$. 
For $v\in B_{N/KL, i} \cap B_{K'L', i'}$ (where $i=1, \ldots, (KL)^d$ and $i' = 1, \ldots, (N/K'L')^d$), define
\begin{equation}\label{eq-xi-MBRW}
 {\xi}_{N, v, \mathrm{MBRW}}= \sum_{j=\ell' + k'}^{n-k-\ell} \sum_{B \in \mathcal{B}_j (v_{K'L', i'})} b^N_{ i, j, B} \,.
 \end{equation}
 Note that by our construction $\{\xi_{N, v, \mathrm{MBRW}}: v\in B_{N/KL, i}\}$ are independent of each other for $i=1, \ldots, (KL)^d$, and in addition $\xi_{N, \cdot, \mathrm{MBRW}}$ is constant over each $K'L'$-box. 
Further, let $\{ \xi_{N, v}^c : v \in V_N \}$, $\{ \xi_{N, v}^b :
v \in V_N \}$ and $\{\xi_{N, v, \mathrm{MBRW}}: v\in V_N\}$ be
independent of each other. One has by Assumption (A.1) that
$$|\var (\xi_{N, v}^c + \xi_{N, v}^b + \xi_{N, v, \mathrm{MBRW}}) - 
\var \varphi_{N, v}| \leq 4\alpha\,.$$
Let $a_{N, v}$ be a sequence of numbers such that for all 
$v\in B_{N/KL,i}$ and all $1\leq i\leq (KL)^d$, 
\begin{equation}
	\label{eq-paris6}
	\var ( \xi_{N, v}^c + \xi_{N,  v}^b + \xi_{N, v, \mathrm{MBRW}}) + a^2_{N, v}  = \var \varphi_{N, v} + 4\alpha\,.
\end{equation}
(Here, the sequence $a_{N, v}$ implicitly depends on $(KL)$.)
It is clear that 
\begin{equation}\label{eq-bound-a-N-v}
\max_{v\in V_N} a_{N, v} \leq \sqrt{8\alpha}\,.
\end{equation}
For $ v\in B_{N/KL, i} \mbox{ and } v \equiv \bar v \mod K'L'$, one has
\begin{align*}
a_{N, v}^2 & = \var \varphi_{N, v} + 4 \alpha - \var\varphi_{KL, w_i} - \var \varphi_{K'L', \bar v} - \log(\tfrac{N}{KLK'L'})\\
 &= \log N - \log (KL) + \epsilon_{N, K L, K' L' } + 4 \alpha - 
 \var \varphi_{K'L', \bar v} - \log(\tfrac{N}{KLK'L'})\geq 0,
\end{align*}
where, by Assumptions (A.2),
\begin{equation}\label{eq-epsilon-N-K-L}
\limsup_{(L,K,L',K') \Rightarrow \infty} \limsup_{N\to \infty} 
\epsilon_{N, K L, K' L'} = 0\,.
\end{equation}
Therefore, one can write 
\begin{equation}\label{eq-a-decomposition}
a_{N,v}^2 = a^2_{K', L', \bar v} + \epsilon_{N, K L, K' L'}\,, 
\end{equation}
where $a_{K' L', \bar v}$ depends on $(K' L', \bar v)$.
By Assumption (A.2) and the continuity of $f$, one has
$$\limsup_{(L,K,L',K') \Rightarrow \infty}  
\sup_{u,v: \|u-v\|_\infty\leq L'}
\limsup_{N\to \infty} |\var \xi_{N, v}^b - \var \xi_{N, u}^b| = 0.$$ 
Therefore, we can further require that
\begin{equation}\label{eq-continuity-a}
|a_{K' L', \bar v} - a_{K' L' ,  \bar u}| \leq \epsilon_{N, K L, K' L'} 
\mbox{ for all } \|\bar v - \bar u\|_\infty \leq L'\,.
\end{equation}
Let $\phi_j$ be i.i.d.\ standard Gaussian variables. For 
$ v\in  B_{K'L', j} \mbox{ and } v \equiv \bar v \mod K'L'$, define 
\begin{equation}\label{eq-def-xi}
\xi_{N, v} =  \xi_{N, v}^c + \xi_{N,  v}^b + \xi_{N, v, \mathrm{MBRW}}+ 
a_{K'L', \bar v} \phi_j\,.
\end{equation}
It follows from \eqref{eq-paris6} and \eqref{eq-a-decomposition} that
\begin{equation}\label{eq-var-approx}
\limsup_{(L,K,L',K') \Rightarrow \infty} 
\limsup_{N\to \infty} |\var \xi_{N, v} - \var \varphi_{N, v} - 4 \alpha| = 0\,.
\end{equation}
Finally, we partition $V_N$ into a disjoint union of boxes of side 
length $N/L$ which we denote by 
$\mathcal B_{N/L} = \{B_{N/L, i}: 1\leq i\leq L^d\}$, as well as
a disjoint union of boxes of side length $L$ which we denote by 
$\mathcal B_L = \{B_{L, i}: 1\leq i\leq (N/L)^d\}$. Again, we denote by 
$v_{N/L, i}$ and $v_{L, i}$ the left bottom corner of the boxes 
$B_{N/L, i}$ and $B_{L, i}$, respectively.

For $\delta>0$ and any box $B$, denote by $B^\delta\subseteq B$ the 
collection of all vertices in $B$ that are $\delta \ell_B$ away from 
its boundary $\partial B$ (here $\ell_B$ is the side length of $B$). 
Let 
$$V^*_{N, \delta} = (\cup_i B^\delta_{N/L, i} )\cap (\cup_i B^\delta_{N/KL, i}) \cap ( \cup_i B_{L, i}^\delta) \cap (\cup_i B_{KL, i} )\,.$$
One has $|V^*_{N, \delta}| \geq (1-100d\delta) |V_N|$. 

The following lemma suggests that $\{\xi_{N, v}: v\in V_N\}$ is a good approximation of $\{\varphi_{N, v}: v\in V_N\}$.
\begin{lemma}\label{lem-xi-approx}
Let 
Assumptions (A.1), (A.2) and (A.3) hold. Then there 
exist $\epsilon'_{N, K, L, K', L' }>0$ with 
$
\limsup_{(L,K,L',K') \Rightarrow \infty} 
\limsup_{N\to \infty} \epsilon'_{N, K, L, K', L'} = 0\,,
$
such that the following hold for all $u, v\in V^*_{N, \delta}$ :

\noindent (a) If $u, v\in B_{L', i}$ for some $1\leq i\leq (N/L')^d$, then $|\E(\xi_{N, u} - \xi_{N, v})^2 - \E (\varphi_{N, u} - \varphi_{N, v})^2 |\leq \epsilon'_{N, K, L, K', L' }$.

\noindent (b) If $u \in B_{N/L, i}$, $v\in B_{N/L, j}$ with $i\neq j$, then $|\E \xi_{N, u} \xi_{N, v} - \E\varphi_{N, v} \varphi_{N, u}| \leq \epsilon'_{N, K, L, K', L' }$.

\noindent (c) Otherwise, $|\E \xi_{N, u} \xi_{N, v} - \E\varphi_{N, v} \varphi_{N, u}| \leq 4\log (1/\delta) + 40\alpha$.
\end{lemma}
\begin{proof}
\noindent {\bf (a): } Let $i'$ be such that $B_{L', i} \subseteq 
B_{K'L', i'}$. By \eqref{eq-continuity-a} and \eqref{eq-def-xi}, one has
$$|\E (\xi_{N, u} - \xi_{N, v})^2  -  \E (\varphi_{KL, u - v_{KL, i'}} - 
\varphi_{KL, v - v_{KL, i'}})^2| \leq 4 \epsilon_{N, K L, K' L'}\,,$$
where $\epsilon_{N, K L, K' L'}$ satisfies \eqref{eq-epsilon-N-K-L} 
(and was defined therein).
By Assumption (A.2), one has
$$\limsup_{(L,K,L',K') \Rightarrow \infty} 
\limsup_{N\to \infty} |\E (\varphi_{KL, u - v_{KL, i'}} - 
\varphi_{KL, v - v_{KL, i'}})^2- \E (\varphi_{N, u} - \varphi_{N, v})^2| =0\,.$$
Altogether, this completes the proof for (a).

\noindent {\bf (b): } 
Let $i', j'$ be such that $u\in B_{N/KL, i'}$ and $v\in B_{N/KL, j'}$, and
assume w.l.o.g. that
$K'\gg L' \gg K \gg L\gg 1/\delta$. The
definition of $\{\xi_{N, v}\}$ gives 
$$\E \xi_{N, v} \xi_{N, u} = \E \varphi_{KL, w_{i'}} \varphi_{KL, w_{j'}}$$
where $w_{i'} = \frac{v_{N/KL, i'}}{N/KL}$ and $w_{j'} = 
\frac{v_{N/KL, j'}}{N/KL}$. In this case, we have 
$|w_{i'} - w_{j'}| \geq \delta K$. Writing
$x_u = u/N, x_v = v/N$ and $y_u = w_{i'}/KL, y_v = w_{j'}/KL$, one obtains
$$|y_u - y_v| \geq \delta/L, |x_u - x_v| \geq \delta /L, 
|x_u - y_u|\leq 1/K, |x_v - y_v| \leq 1/K\,.$$
Therefore, Assumption (A.3) yields  $\limsup_{(L,K,L',K') 
\Rightarrow \infty} \limsup_{N\to \infty} |\E \xi_{N, u} \xi_{N, v} - 
\E\varphi_{N, v} \varphi_{N, u}| =0$,
completing the proof of (b).

\noindent {\bf (c).} In this case, one has
\begin{align*}
\E \xi_{N, v}
 \xi_{N, u} =& \E \xi_{N, v }^c \xi_{N, u}^c + \E \xi_{N,  v}^b \xi_{N, u}^b + \E \xi_{N, u, \mathrm{MBRW}} \xi_{N, v, \mathrm{MBRW}} +  \mathrm{err_1} \\
 = &\log KL - \log_+ (\tfrac{|u-v|}{N/KL}) \\
  &+  \mathbf{1}_{|u-v| \leq N/KL}(\log \tfrac{N}{(KLK'L')} - \log_+ \tfrac{|u-v|}{K'L'}) + \mathrm{err_2} \\
  =& \log N - \log_+ |u-v| + \mathrm{err_2},
\end{align*}
where  $|\mathrm{err_1}|\leq 8\alpha$ and $|\mathrm{err_2}| \leq 2\log 1/\delta + 20\alpha$. Combined with Assumption (A.1), this completes the proof 
of (c) and hence of the lemma.
\end{proof}
  
\begin{lemma}\label{lem-key-convergence}
Let Assumptions (A.0), (A.1), (A.2) and (A.3) hold. Then,
$$ \limsup_{(L,K,L',K') \Rightarrow \infty}  \limsup_{N\to \infty} 
d(M_N - m_N,  \max_{v\in V_N} \xi_{N, v} - m_N - 2\alpha \sqrt{2d} ) = 0\,.$$
\end{lemma}
\begin{proof}
By Proposition~\ref{prop-maxrtfluc}, 
it suffices to show that for all $\delta>0$
$$ \limsup_{(L,K,L',K') \Rightarrow \infty} 
\limsup_{N\to \infty} \limsup_{N\to \infty} 
d(\max_{v\in V_{N, \delta}^*}  \varphi_{N, v}- m_N,  
\max_{v\in V_{N, \delta}^* }\xi_{N, v} - m_N - 2\alpha \sqrt{2d} ) = 0\,.$$
Consider a fixed $\delta>0$. Let $\sigma_*^2 = 4\log (1/\delta) + 60\alpha$. 
Let $\sigma_{\mathrm{lw}} = (0, \sqrt{\sigma_*^2 + 4\alpha} )$ and 
$\sigma_{\mathrm{up}} = (\sigma^*, 0)$. Define 
$\{\tilde \varphi_{N, L', L, \sigma_{\mathrm{lw}}, v} : v\in V_N\}$ as in 
\eqref{eq-def-tilde-phi} with  $r_1 = L'$, $r_2 = L$ and 
$\sigma = \sigma_{\mathrm{lw}}$. Analogously, define $\{\tilde 
	\xi_{N, L',  L, \sigma_{\mathrm{up}}, v}: v\in V_N\}$. 
	By \eqref{eq-def-xi} and Lemma~\ref{lem-xi-approx}, one has
	for all $u, v \in V_{N, \delta}^*$,
\begin{align*}
|\var \tilde \varphi_{N, L', L, \sigma_{\mathrm{lw}}, v} - 
\var  \tilde \xi_{N, L', L, \sigma_{\mathrm{up}}, v}| 
&\leq \bar \epsilon_{N, K, L, K', L'}\,, \\
\E \tilde \xi_{N, L, \sigma_{\mathrm{up}}, v} \tilde 
\xi_{N, L, \sigma_{\mathrm{up}}, u} & \leq \E \tilde \varphi_{N, L, 
	\sigma_{\mathrm{lw}}, v}\tilde \varphi_{N, L, 
		\sigma_{\mathrm{lw}}, u} + \bar \epsilon_{N, K, L, K', L'}\,,
\end{align*}
where $\limsup_{(L,K,L',K') \Rightarrow \infty} 
\limsup_{N\to \infty} \bar \epsilon_{N, K, L, K', L'} = 0$. 
Since $\{\tilde \varphi_{N, L', L, \sigma_{\mathrm{lw}}, v} : 
v\in V^*_{N, \delta}\}$ satisfies Assumption (A.1) with $\alpha$ 
being replaced by $10\alpha + \sigma_*^2$, one may
apply Lemma~\ref{lem-epsilon-domination} and obtain that
$$\limsup_{(L,K,L',K') \Rightarrow \infty}\limsup_{N\to \infty} 
\tilde d (\max_{v\in V_{N, \delta}^*} \tilde \varphi_{N, L',  L, 
	\sigma_{\mathrm{lw}}, v} - m_N, 
	\max_{v\in V_{N, \delta}^*} \tilde \xi_{N, L',  
		L, \sigma_{\mathrm{up}}, v} - m_N) = 0\,.$$
By Lemma~\ref{lem-nfmax} (it is clear that the same statement holds for maximum over $V_{N, \delta}^*$), one gets
\begin{align*}
&\limsup_{(L,K,L',K') \Rightarrow \infty} \limsup_{N\to \infty}  
d(\max_{v\in V_{N, \delta}^*} \tilde 
\varphi_{N, L',  L, \sigma_{\mathrm{lw}}, v} - m_N - 
(\sigma_*^2 + 4\alpha) \sqrt{d/2} , \max_{v\in V_{N, \delta^*}} \varphi_{N, v} - m_N) = 0\,,\\
&\limsup_{(L,K,L',K') \Rightarrow \infty}  \limsup_{N\to \infty} 
d(\max_{v\in V_{N, \delta}^*} \tilde \xi_{N, L', L, \sigma_{\mathrm{up}}, v} 
- m_N - (\sigma_*^2) \sqrt{d/2} , 
\max_{v\in V_{N, \delta^*}} \xi_{N, v} - m_N) = 0\,.
\end{align*}
Altogether, this gives that 
$$ \limsup_{(L,K,L',K') \Rightarrow \infty} \limsup_{N\to \infty} 
\tilde d(\max_{v\in V_{N, \delta}^*} \varphi_{N, v}- m_N, 
\max_{v\in V_{N, \delta}^* }\xi_{N, v} - m_N - 2\alpha \sqrt{2d} )= 0\,.$$
The other direction
of stochastic domination follows in the same manner. 
Altogether, this completes the proof of the lemma.
\end{proof}

\subsection{Convergence in law for the centered maximum}\label{sec:limit-law}

In light of Lemma~\ref{lem-key-convergence}, in order to prove 
Theorem~\ref{thm-convergence} it remains to show the convergence 
in law for the centered maximum of $\{\xi_{N, v}: v\in V_N\}$. 
To this end, we will follow the proof of the convergence in law in 
the case of 
the 2D DGFF given in \cite{BDZ13}. Let the \textit{fine field} be defined
as 
$\xi^f_{N, v} = \xi_{N, v} - \xi^c_{N, v}$, and note that it implicitly
depends on  $K'L'$.
As in \cite{BDZ13}, a key step in the proof of convergence of
the centered maximum is the following sharp tail estimate on the right
tail of the distribution of
$\max_{v\in B} \xi_{N, v}^f$ for $B\in \mathcal B_{N/KL}$. The proof of
this estimate is postponed to the appendix.
\begin{prop}\label{prop-rttailrefined}
Let Assumptions (A.1), (A.2) and (A.3) hold. Then
there exist constants $C_\alpha, c_\alpha>0$ depending only on $\alpha$ and
constansts 
$c_\alpha\leq 
\beta_{K', L'}^*\leq C_\alpha$ such that
\begin{equation}\label{eq-righttail-finefield}
\lim_{z \rightarrow \infty} \limsup_{L'\to \infty}\limsup_{K' \rightarrow 
\infty} \limsup_{N \rightarrow \infty} |z^{-1} e^{\sqrt{2d}z} 
\P(\max_{v \in B_{N/KL, i} } \xi_{N,v}^f  \ge m_{N/KL} + z) - 
\beta_{K', L'}^* |=0.
\end{equation}
\end{prop}

\begin{remark}\label{remark-analogue}
 Proposition~\ref{prop-rttailrefined} is analogous to 
 \cite[Proposition 4.1]{BDZ13}, but there are two important
 differences: 
 \begin{enumerate}
	 \item
In Proposition~\ref{prop-rttailrefined} the convergence is
to a constant $\beta_{K', L'}^*$ which depends on $K', L'$,
while in \cite[Proposition 4.1]{BDZ13} the convergence is to
an absolute constant $\alpha^*$. This is because the
fine field $\xi_{N, v}$ here 
implicitly depends on $K', L'$, and thus a priori one is
not able to eliminate the dependence on $(K', L')$ from the limit.
However, in the same spirit as in \cite{BDZ13},
the dependence on $(K', L')$ is not an issue for deducing a 
convergence in law  --- the crucial requirement is the independence of $N$. 
Eventually, 
we will deduce the 
convergence of $\beta^*_{K', L'}$  as $K', L'\to \infty$ in that 
order from the convergence in law of the centered maximum. 
\item
In  \cite[Proposition 4.1]{BDZ13}, one also controls
the limiting distribution of the location of the 
maximizer while in Proposition~\ref{prop-rttailrefined} this is 
not mentioned. This is because in the current situation and unlike the 
construction in \cite{BDZ13}, the
coarse field $\{\xi^c_{N, v}\}$ is constant over each box $B_{N/KL, i}$, 
and thus the location of the maximizer of the fine field in each of the
boxes $B_{N/KL, i}$ is irrelevant to the value of the maximum for 
$\{\xi_{N, v}\}$.
\end{enumerate}
\end{remark}

Next, we construct the
limiting law of the centered maximum of $\{\xi_{N, v}: v\in V_N\}$.
We partition $[0,1]^d$ into $R = (KL)^d $ disjoint boxes of equal sizes. 
Let  $\beta^*_{K', L'}$ be as defined in the statement of 
Proposition~\ref{prop-rttailrefined}. By that proposition,
there exists a function $\gamma : \mathbb R \mapsto \mathbb R$  
that grows to infinity arbitrarily slowly (in particular, we may
assume that $\gamma(x)\leq \log\log\log x$) such that 
$$\lim_{z'\to \infty}\limsup_{L'\to \infty} 
\limsup_{K'\to \infty} \limsup_{N\to \infty}
\sup_{z'\leq z\leq \gamma(K'L')} 
|z^{-1} e^{\sqrt{2d}z} \P(\max_{v \in B_{N/KL, i} } 
\xi_{N,v}^f  \ge  m_{N/KL} + z) - \beta_{K', L'}^*| =0.$$

Let $\{{\varrho}_{R, i}\}_{i=1}^{R}$ be independent Bernoulli random 
variables with 
$$\P(\varrho_{R, i} = 1)=\beta^*_{K', L'} 
\gamma(KL) e^{-\sqrt{2d}\gamma(KL)}\,.$$
In addition, consider independent random 
variables $\{Y_{R, i} \}_{i=1}^{R}$ such that
\begin{equation}\label{eq-distrY}
\P(Y_{R, i} \geqslant x )=
\tfrac{\gamma(KL) + x}{\gamma(KL)}e^{-\sqrt{2d}x}\qquad x \geqslant 0.
\end{equation}
Let $\{Z_{R, i}: 1\leq i\leq R\}$ be an independent Gaussian field 
with covariance matrix $\Sigma$ (recall that $\Sigma$ is 
of dimension $R\times R$). 
We then define 
$$G_{K, L, K', L'}^* = \max_{1\leq i\leq R, \varrho_{R, i} = 1} G_{R, i} 
\mbox { where } G_{R, i} = 
\varrho_{R, i} (Y_{R, i} + \gamma(KL) )+ Z_{R, i} - \sqrt{2d} \log (KL) $$
(here we use the convention that $\max \emptyset = 0$).
Let  $\bar \mu_{K,L, K', L'}$ be the distribution of $G_{K, L, K', L'}^*$. We note that  $\bar {\mu}_{K, L, K', L'}$  does not depend  on $N$. 
\begin{theorem}\label{thm-conv2}
Let Assumptions (A.0), (A.1), (A.2) and (A.3) hold. Then,
\begin{equation}\label{eq-conv2}
 \limsup_{(L,K,L',K') \Rightarrow \infty}
 \limsup_{N \rightarrow \infty } d(\mu_N,  \bar{\mu}_{K,L, K', L'})=0,
\end{equation}
where $\mu_N$ is the law of $\max_{v\in V_N} \xi_{N, v} - m_N $.
\end{theorem}
(Note that $\mu_N$ does depend on $KL,K'L'$.)

\begin{proof} [Proof of Theorem~\ref{thm-convergence}] Theorem~\ref{thm-convergence} follows from Lemma~\ref{lem-key-convergence} and Theorem~\ref{thm-conv2}. 
\end{proof}

Next, we give the proof of Theorem~\ref{thm-conv2}. Our proof is conceptually simpler than that of its analogue \cite[Theorem 2.4]{BDZ13}, since our coarse field is constant over a box of size $N/KL$ (and thus no consideration of the location for the maximizer in the fine field is needed).

\begin{proof}[Proof of Theorem~\ref{thm-conv2}]
Denote by $\tau = \arg\max_{v\in V_N} \xi_{N, v}$. 
Applying Theorem~\ref{thm-tightness} to the Gaussian fields 
$\{\xi_{N, v}: v\in V_N\}$ and $\{\xi_{N, v}^c: v\in V_N\}$
(where the maximum of $\{\xi_{N, v}^c: v\in V_N\}$ is 
equivalent to the maximum of a log-correlated Gaussian field in a 
$KL$-box), we deduce that 
\begin{equation}\label{eq-except-prob}
 \limsup_{(L,K,L',K') \Rightarrow \infty} 
 \limsup_{N \rightarrow \infty } \P(\xi_{N, \tau}^f \geq m_{N/KL} 
 + \gamma(KL)+1)=1\,.
\end{equation}
Therefore, in what follows, we assume w.l.o.g.  the occurrence of the event 
$$\{\xi_{N, \tau}^f \geq \sqrt{2d} \log \tfrac{N}{KL} - \tfrac{3}{2\sqrt{2d}} \log \log \tfrac{N}{KL} + \gamma(KL)+ 1\}\,.$$ 
Let $\mathcal E = \cup_{1\leq i\leq R}\{\max_{v\in B_{N/KL, i}}\xi_{N, v}^f \geq m_{N/KL} + KL +1\}$. 
A simple union bound over $i$ gives that 
\begin{equation}
  \label{eq-loststar}
  \limsup_{(L,K,L',K') \Rightarrow \infty} 
  \limsup_{N \rightarrow \infty } \P(\mathcal E) = 0\,.
\end{equation}
Thus in what follows we assume without loss that $\mathcal E$ does not occur. Analogously, we let $\mathcal E' = \cup_{1\leq i\leq R}\{Y_{R, i} \geq  KL +1 - \gamma (KL)\}$. We see from the  union bound that 
\begin{equation}
  \label{eq-losttwostar}
  \limsup_{(L,K,L',K') \Rightarrow \infty} \limsup_{N \rightarrow \infty } \P(\mathcal E') = 0\,.
\end{equation}
In what follows, we assume without loss that $\mathcal E'$ does not occur.

For convenience of notation, we denote by 
$$M^f_{N, i} = \max_{v\in B_{N/KL, i}} \xi_{N, v}^f  - ( m_{N/KL} +\gamma(KL)  )
\,.$$
By Proposition~\ref{prop-rttailrefined}, there exists $\epsilon^* = \epsilon^* (N, K, L, K', L')$ with  
$$\limsup_{(L,K,L',K') \Rightarrow \infty} 
\limsup_{N \rightarrow \infty } \epsilon^*(N, K, L, K', L') = 0\,,$$ 
such that for some $|\epsilon^\diamond| \leq \epsilon^*/4$
\begin{align*} 
\P(\epsilon^\diamond \leq M_{N, i}^f \leq KL - \gamma(KL) + 1) = \P(\varrho_{R, i} = 1, Y_{R, i}\leq KL - \gamma(KL) +1)\,,
\end{align*}
and that for all $-1\leq t \leq KL-\gamma(KL) + 1$
$$ \P(\varrho_{R, i=1}, Y_{R, i} \leq t - \epsilon^*/2) \leq \P(\epsilon^\diamond \leq M_{N, i}^f \leq t) \leq \P(\varrho_{R, i=1}, Y_{R, i} \leq t+\epsilon^*/2)\,.$$
Therefore,  there exists a coupling between $\{M^f_{N, i}: 1\leq i\leq R\}$ and $\{\varrho_i, Y_{R, i}: 1\leq i\leq R\}$ such that 
on the event $(\mathcal E \cup \mathcal E')^c$,
\begin{equation}\label{eq-coupling-fine-field}
\varrho_{R, i} =1, |Y_{R, i} - M^f_{N, i}| \leq \epsilon^* 
\mbox{ if } M^f_{N, i} \geq  \epsilon^*\,,\mbox{ and }  |Y_{R, i} - M^f_{N, i}| 
\leq \epsilon^*  \mbox{ if } \varrho_{R, i} = 1\,.
\end{equation}
In addition, it is trivial to couple such that $\xi_{N, v}^c = Z_{R, i}$ for all $v\in B_{N/KL, i}$ and $1\leq i\leq R$. Also, notice the following simple fact
$$ \limsup_{L\to \infty}\limsup_{K\rightarrow \infty} \limsup_{N \rightarrow \infty } (m_N -  m_{N/KL}- \sqrt{2d} \log (KL)) =0 \,.$$
Altogether, we conclude that  there exists a coupling such that outside an event
of probability tending to 0 as $N\to \infty$ and then $(L,K,L',K') 
\Rightarrow \infty$ (c.f. \eqref{eq-except-prob}, \eqref{eq-loststar}, 
\eqref{eq-losttwostar}) we have 
$$\max_{v\in V_N} (\xi_{N, v} -m_N) - G_{K, L, K', L'}^*\leq 2\epsilon^*\,.$$
Now, let $\tau' = \arg\max_{1\leq i\leq R} G_{R, i}$. Applying Theorem~\ref{thm-tightness} to the Gaussian field $\{Z_{R, i}\}$ and using the preceding inequality, we see that
\begin{equation}\label{eq-tau-prime}
\limsup_{(L,K,L',K') \Rightarrow \infty} \limsup_{N \rightarrow \infty } \P(\varrho_{R, \tau'} = 1) = 1\,.
\end{equation}
Combined with \eqref{eq-coupling-fine-field}, this yields that
 there exists a coupling such that except with probability tending to 
 0 as $N\to \infty$ and then $(L,K,L',K') \Rightarrow \infty$ we have 
$$|\max_{v\in V_N} ( \xi_{N, v} -m_N) - G_{K, L, K', L'}^*|\leq 2\epsilon^*\,.$$
thereby completing the proof of Theorem~\ref{thm-conv2}.
\end{proof}

\begin{proof}[Proof of Theorem \ref{explicit-limit-law}]
Recall that $G^*_{K,L, K', L'}$ is a random
variable with law $\bar \mu_{K, L, K', L'}$. 
We will construct
random variables $\mathcal Z_{K, L}$,
measurable with respect to 
${\cal F}^c:=\sigma(
\{Z_{R, i}\})$, so that
\begin{equation}
	\label{eq-gumbel1}
\limsup_{(L,K,L',K') \Rightarrow \infty}
\frac{\bar \mu_{K,L, K', L'}( (-\infty,x] )}{
	\E(\text{e}^{- \beta_{K', L'}^* \mathcal Z_{K,L} \text{e}^{-\sqrt{2 d}x}})} = \liminf_{(L,K,L',K') \Rightarrow \infty}
\frac{\bar \mu_{K,L, K', L'}( (-\infty,x] )}{
	\E(\text{e}^{- \beta_{K', L'}^* \mathcal Z_{K,L} \text{e}^{-\sqrt{2 d}x}})} = 1\,.
\end{equation}
for all $x$. To demonstrate \eqref{eq-gumbel1}, 
due to
\eqref{eq-tau-prime}, we may and will
assume without loss that $\varrho_{R, \tau'} = 1$. 
Define $S_{R, i} := \sqrt{2d}\log (KL) - Z_{R, i}$. 
Then,
for any real $x$,
\begin{equation}
	\label{eq-gumbel2}
	\P(G^*_{K,L, K', L'}\leq x)= 
	\E\left(\prod_{i=1}^{R}\left(1-
	\P(\varrho_{R, i}Y_{R, i}>S_{R,i} + x - \gamma(KL)\,|\,{\cal F}^c)
	\right)\right)\,.
\end{equation}
In addition, the union bound gives that 
$$\limsup_{KL\to \infty}\P(\mathcal D) = 1 \mbox{ where } \mathcal D = 
\{\min_{1\leq i\leq R} S_{R,i} \geq   2\gamma(KL)\}\,.$$
So in the sequel we assume that $\mathcal D$ occurs. 
By the
definition of $\varrho_{R, i}$ and $Y_{R, i}$, we get that
\begin{align*}
  \P&(\varrho_{R, i}Y_{R, i}>S_{R,i} + x - \gamma(KL)\,|\,{\cal F}^c) 
=
\beta^*_{K', L'}(S_{R,i} + x) e^{-\sqrt{2d} (S_{R,i}
 + x)} \to 0 \mbox{ as } KL\to \infty\,.
 \end{align*}
	Therefore, 
	\begin{align}\label{eq-prob-approx}
		\exp(-(1+\epsilon_{K,L})
	\beta_{K', L'}^* 
	S_{R, i} &
	\text{e}^{-\sqrt{2 d}(x + S_{R, i})})\,\leq \,
	\P(\varrho_{R, i}Y_{R, i}\leq S_{R, i}+ x - \gamma(KL)\,|\,{\cal F}^c) \nonumber\\
	& \leq \exp(-(1-\epsilon_{K,L})
	\beta_{K', L'}^* 
	S_{R, i}
	\text{e}^{-\sqrt{2 d}(x + S_{R, i})})\,,
\end{align}
for $\epsilon_{K,L} > 0$ with 
$$\limsup_{KL\to\infty}
\epsilon_{K,L}= 0.$$
	Define $\mathcal Z_{K,L}=\sum_{i=1}^{R}S_{R, i}e^{-
		\sqrt{2d} S_{R, i}}$ (this is the analogue of a
		derivative martingale, see 
		\eqref{eq-def-derivative-martingale}). 
		Substituting  \eqref{eq-prob-approx} into \eqref{eq-gumbel2}
		 completes the proof of \eqref{eq-gumbel1}. 
Now, combining \eqref{eq-gumbel1} and Theorem~\ref{thm-conv2}, we see that we necessarily have 
$$\limsup_{K'\to \infty} \limsup_{L' \to \infty} |\beta^*_{K', L'} - \beta^*| = 0$$
for a number $\beta^*$ that does not depend on $(K', L')$. Plugging the preceding inequality into \eqref{eq-gumbel1}, we deduce that
\begin{equation}
	\label{eq-gumbel}
\limsup_{(L,K,L',K') \Rightarrow \infty}
\frac{\bar \mu_{K,L, K', L'}( (-\infty,x] )}{
	\E(\text{e}^{- \beta^* \mathcal Z_{K,L} \text{e}^{-\sqrt{2 d}x}})} = \liminf_{(L,K,L',K') \Rightarrow \infty}
\frac{\bar \mu_{K,L, K', L'}( (-\infty,x] )}{
	\E(\text{e}^{- \beta^* \mathcal Z_{K,L} \text{e}^{-\sqrt{2 d}x}})} = 1\,.
\end{equation}
Combining \eqref{eq-gumbel} with Theorem~\ref{thm-conv2} again, we see that $\mathcal Z_{K, L}$ converges weakly to a random variable $\mathcal Z$ as $K\to \infty$ and then $L\to \infty$. Also note that $\mathcal Z_{K, L}$ depends only on the product $KL$. Therefore, this 
implies that $\mathcal Z_N$ converges weakly to a random variable 
$\mathcal Z$. From the tightness of the laws
$\bar \mu_{K, L, K', L'}$, it follows that
$\mathcal Z>0$ a.s. This completes the proof of 
Theorem~\ref{explicit-limit-law}.
\end{proof}

\begin{proof}[Proof of Remark \ref{rem-universal}]
  Consider two sequences $\{\varphi_{N,v}\}$ and 
  $\{\tilde \varphi_{N,v}\}$ that satisfy assumptions (A.0)--(A.3) 
  with the same functions $h(x,y)$ and $f(x)$ but possibly different
  functions $g(u,v), \tilde g(u,v)$ and different constants $\alpha^{(\delta)},
  \alpha^{(\delta),'}$ and $\alpha_0,\alpha_0'$. 
  Introduce the corresponding
  fields 
  $$\xi_{N,KL,K'L'}=\xi_{N,KL,K'L'}^c+ \xi_{N,KL,K'L'}^f\,,\quad
  \tilde \xi_{N,KL,K'L'}=\tilde \xi_{N,KL,K'L'}^c+ \tilde 
  \xi_{N,KL,K'L'}^f\,, 
  $$
  see Section \ref{sec:decomposition}.
  Set also
  $$\hat \xi_{N,KL,K'L'}=\tilde \xi_{N,KL,K'L'}^c+ \xi_{N,KL,K'L'}^f\,.$$
  Let $\nu_N, \tilde \nu_N$ denote the 
  laws of the centered maxima 
  $\max_{v\in V_N} \varphi_{N,v}-m_N$, 
  $\max_{v\in V_N} \tilde \varphi_{N,v}-\tilde m_N$,  
  and let $\mu_N,\tilde \mu_N,\hat \mu_N$ denote the laws of 
  the centered maxima
  of the $\xi_N,\tilde\xi_N,\hat\xi_N$ fields. 
  (Recall that the latter depend also
  on $KL, K'L'$ but we drop that fact from the notation.) 
  By  Lemma~\ref{lem-key-convergence}, we have
  \begin{equation}
	  \label{eq-rem1}
	  \limsup_{(L,K,L',K') \Rightarrow 
	  \infty} \limsup_{N\to\infty}
	  \left(d(\mu_N,\nu_N)+d(\tilde \mu_N,\tilde \nu_N)\right)
	  =
	  0\,.
\end{equation}

For $s\in \mathbb{R}$, 
let $\theta_s\mu$ denote the shift of a probability measure $\mu$ on 
$\mathbb{R}$, that
is $\theta_s\mu(A)=\mu(A+s)$ for any measurable set $A$.
Recall the construction of $\bar \mu_{K,L,K',L'}$, see
Theorem \ref{thm-conv2}, and construct similarly 
$\tilde{\mu}_{K,L,K',L'}$ and
$\hat \mu_{K,L,K',L'}$. 
Note that, by construction, there exists $s=s(KL)$, bounded uniformly in
$KL$, so that  $\theta_s \hat \mu_{K,L,K',L'}=\tilde \mu_{K,L,K',L'}$.
In particular, from Theorem \ref{thm-conv2} we get that
  \begin{equation}
	  \label{eq-rem2}
	  \limsup_{(L,K,L',K') \Rightarrow 
	  \infty} \limsup_{N\to\infty}
	 \left(d(\mu_N,\bar \mu_{K,L,K',L'})+d(\tilde \mu_N,\theta_s\hat \mu_{K,L,K',L'})\right)=0\,.
       \end{equation}
       From \eqref{eq-rem1} and \eqref{eq-rem2}, one can find a sequence
       $L(N),K(N),K'(N),L'(N)$ along which the convergence still holds
       (as $N\to\infty$). Let $\{\eta_{v,N}\}$ and $\{\hat\eta_{v,N}\}$
       denote the fields $\{\xi_{v,N}\}$ and $\{\hat \xi_{v,N}\}$ 
       with this choice of
       parameters, and let $\bar\mu_N$ and $\hat\mu_N$ 
       denote the corresponding laws of
       the maximum. Let $\mu_\infty$, $\tilde \mu_\infty$ denote
       the limits of $\mu_N$ and $\tilde\mu_N$, which exist by theorem
       \ref{thm-convergence}. From the above considerations we have 
       that $\bar\mu_N\to \mu_\infty$ and $\theta_{s(N)}\hat\mu_N\to \tilde
       \mu_{\infty}$. On the other hand,
       the fields $\eta_{N,\cdot}$ and $\hat \eta_{N,\cdot}$ both 
       satisfy assumptions (A.0)-(A.3) with the same functions $f,g,h$ and thus,
       interleaving between then one deduces that the laws of 
       their centered maxima converge to the same limit,
       denoted $\Theta_\infty$. It follows that necessarily, $s(N)$ 
       converges and $\mu_\infty=\theta_s \tilde\mu_\infty=\Theta_\infty$. 
       Using the characterization in Theorem
       \ref{explicit-limit-law}, this yields the claim in the remark.
\end{proof}
\section{An example: the circular logarithmic REM}
\label{sec:example}
In the important paper \cite{BF}, the authors introduce a one dimensional
logarithmically correlated Gaussian field, which they call the 
\textit{circular logarithmic REM} (CLREM). Fyodorov and Bouchaud consider
the CLREM  as a  prototype
for Gaussian fields exhibiting Carpentier-LeDoussal 
freezing. (We do not discuss in this paper the notion
of freezing, referring instead to \cite{BF} and to \cite{SZ15}.)
Explicitly,  fix an integer $N$, set $\theta_k=2\pi k/N$, and introduce the
matrix
$$R_{k,\ell}=-\frac12 \log \left(4\sin^2\left(\frac{\theta_k-\theta_\ell}{2}
\right)\right){\bf 1}_{k\neq \ell} +(\log N+W){\bf 1}_{k=\ell}\,,$$
where $W$ is a constant independent on $N$. It is not hard to verify
(and this is done explicitly in \cite{BF}) that one can choose $W$ so that
the matrix $R$ is positive definite for all $N$; the resulting Gaussian field 
$\varphi_{N,v}$ with correlation matrix $R$ is the CLERM. One
may think of the CLREM as indexed by $V_N$ in dimension $d=1$, 
or (as the name indicates) by
an equally spaced collection of $N$ points on the unit circle in the complex plane.

Let $M_N=\max_{v\in V_N} \varphi_{N,v}$.
The following is a corollary of Theorems \ref{thm-tightness} and \ref{explicit-limit-law}.
\begin{cor}
	\label{cor-CLREM}
	$\E M_N=\sqrt{2}\log N-(3/2\sqrt{2})\log\log N+O(1)$ and
	there exist a constant $\beta^*$ and 
	a random variable $\mathcal{Z}$ so that 
	\begin{equation}
		\label{eq-paris5}
		\lim_{N\to\infty} \P(M_N-\E M_N\leq x)=\E(e^{-\beta^* 
			\mathcal{Z}e^{-\sqrt{2}x}})\,.
		\end{equation}
\end{cor}
\begin{proof}
	Assumptions (A.0) and (A.1) are immediate to check. 
	An explicit computation reveals that
	Assumption (A.2) holds with $f(x)=0$ and
	$$g(u,v)=\left\{\begin{array}{ll}
		-W, & u=v\\
		\log(4\pi)+|u-v|, & u\neq v
	\end{array}
	\right . .$$
	Finally, it is clear that Assumption (A.3) holds with $h(x,y)=
	\log(4\sin^2(2\pi |x-y|))$. Thus, Theorems
	\ref{thm-convergence} and \ref{explicit-limit-law}
	apply and yields \eqref{eq-paris5}.
\end{proof}
	
\begin{remark}
Remarkably, in \cite{BF} the authors compute explicitly, albeit
nonrigorously, the law of the maximum of the CLREM, up to a deterministic shift
that they do not compute. It was observed in \cite{SZ15} that the law 
computed in \cite{BF} is in fact 
the law of a convolutions of two Gumbel random variables. In the 
notation of Corollary \ref{cor-CLREM}, this means that
one expects that
$2^{-1/2}\log (\beta^* \mathcal{Z})$ is Gumbel distributed.
We do not have a rigorous proof for this claim.
\end{remark}

\medskip
\noindent
{\bf Acknowledgement} 
Remark \ref{rem-universal} answers a question posed to us by
Vincent Vargas. We thank him for the question and for his insights.

\smallskip

\small

\appendix

\section{Proof of Proposition~\ref{prop-rttailrefined}}

Our proof of Proposition~\ref{prop-rttailrefined} is highly similar to the proof in \cite[Proposition 4.1]{BDZ13}, but simpler in a number of places. We will sketch the outline of the arguments, and refer to \cite{BDZ13} extensively (it is helpful to recall Remark~\ref{remark-analogue}). To start, we note that by Lemmas~\ref{lem-righttail} and \ref{lem-slepian}, there exists $c_\alpha>0$ depending only on $\alpha$ such that
\begin{equation}\label{eq-lower-tail-fine-field}
\P(\max_{v\in B_{N/KL, i}} \xi_{N, v}^f \geq m_{N/KL} + z) \geq c_\alpha z e^{-\sqrt{2d} z} \mbox{ for all } 1\leq z\leq  \sqrt{\log N/KL}, 1\leq i\leq (KL)^d\,. 
\end{equation} 
In addition, adapting the proof of \eqref{eq-maxrtflucrttail}, we deduce that there exists $C_\alpha>0$ depending only on $\alpha$ such that
\begin{equation}\label{eq-upper-tail-fine-field}
\P(\max_{v\in B_{N/KL, i}} \xi_{N, v}^f \geq m_{N/KL} + z) \leq C_\alpha z e^{-\sqrt{2d} z} \mbox{ for all } z\geq 1, 1\leq i\leq (KL)^d\,. 
\end{equation} 

Recall the definition of $\{\xi_{N, v}\}$ as in \eqref{eq-def-xi}. In what follows we consider a fixed $i$ and a box $B_{N/KL, i}$. We note that the law of the fine field $\{\xi_{N, v}^f: v\in B_{N/KL, i}\}$ does not depend on $K, L, i$, and hence $\beta^*_{K', L'}$ does not depend on $K, L, i$.  Write $\bar N = N/KL = 2^{\bar n}$ and $\bar L = K'L' = 2^{\bar \ell}$. 
For convenience of notation, we will refer to the box $B_{N/KL, i}$ as $V_{\bar N}$ and let $\Xi_{\bar N}$ be  the collection of all left bottom  corners of $\bar L$-boxes of form $B_{\bar L, j}$ in $B_{N/KL, i}$.
In addition, write $n^* = \frac{\var X_{v, N}}{\log 2}=   \bar n -\bar \ell $, where we denote $X_{v, N} =  \xi_{N, v, \mathrm{MBRW}} $.

For convenience, we now view each $X_{v,N}$ as
the value at time $n^*$ of a Brownian motion with variance rate $\log 2$.
More precisely, we assign to
each Gaussian variable $b_{ j, B}^N$ in
\eqref{eq-xi-MBRW}
an independent Brownian motion, with variance rate $\log 2$,
that runs for $2^{-2j}$ time units and ends at the value $b_{j, B}^N$.
We now define a Brownian motion $\{X_{v,N}(t) : 0\leq t\leq n^*\}$  by concatenating each of the previous Brownian motions associated with $v\in \Xi_{\bar N}$, with earlier times corresponding to larger boxes. From our construction, we see that $X_{v,N}(n^*) = X_{v,N}$. We partition $V_{\bar N}$ into disjoint  $\bar L$-boxes, for which we denote $\mathcal B_{\bar L}$. Further, denote by $B_v$ the $\bar L$-box in $\mathcal B_{\bar L}$ that contains $v$. 
Define
\begin{equation}\label{eq-big-definition}
\begin{split}
E_{v, N}(z) &= \{X_{v,N}(t) \leq z + \frac{ m_{\bar N}}{\bar n}t \mbox{ for all } 0\leq t\leq n^*, \mbox{ and } \max_{u\in B_v} \xi_{u,N}^f \geq m_{\bar N} + z\}\,,\\
F_{v, N}(z) &= \{X_{v,N}(t) \leq z + \frac{m_{\bar N}}{\bar n}t  + 10 (\log (t \wedge (n^*-t)))_+ + z^{1/20}\\
&\qquad\mbox{ for all } 0\leq t\leq n^*, \mbox{ and } \max_{u\in B_v} \xi_{u,N}^f\geq  m_{\bar N} + z\}\,,\\
G_{N}(z) &= \bigcup_{v\in \Xi_{\bar N}}
\bigcup_{0\leq t\leq n^*}\{X_{v,N}(t) >
z+ \frac{ m_{\bar N}}{\bar n}t  + 10 (\log (t \wedge (n^*-t)))_+ + z^{1/20}\}\,.
\end{split}
\end{equation}
Also define
$$\Lambda_{\bar N, z} = \sum_{v\in \Xi_{\bar N}} \one_{E_{v, N}(z) }\,, \mbox{ and }   \Gamma_{\bar N, z} = \sum_{v\in \Xi_{\bar N}}
\one_{F_{v, N}(z) }\,.
$$
In words, the random variable $\Lambda_{N,z}$ counts the number
of boxes in $ {\mathcal B}_{\bar L}$ whose ``backbone'' path $X_{v,N}(\cdot)$
stays below a linear path connecting $z$ to roughly $ m_{\bar N} + z$, so that
one of its ``neighbors'' achieves a terminal value that is at least
$m_{\bar N}+z$; the random variable $\Gamma_{N,z}$ similarly counts boxes in $ {\mathcal B}_{\bar L}$
whose backbone is constrained to stay below a slightly
``upward bent'' curve.
Clearly, $E_{v, N}(z) \subseteq F_{v, N}(z)$ always holds,
as does $\Lambda_{\bar N,z} \le \Gamma_{\bar N,z}$.

By \eqref{eq-def-xi}, for each $v\in \Xi_{\bar N}$ we can write that 
\begin{equation}\label{eq-Y-tail}
\max_{u\in B_v} \xi_{N, v}^f= X_{v, N} + Y_{v, N}\,, 
\end{equation}
where $\{Y_{v, N}\}$ are i.i.d.\ random variables with the same law as $\max_{u\in V_{\bar L}} \varphi_{\bar L, u} + a_{K, L, K', L', u} \phi$ where $\phi$ is a standard Gaussian variable. Crucially, the law of $Y_{v, N}$ does not depend on $N$. In addition, by Proposition~\ref{prop-maxrtfluc}  and Lemma~\ref{lem-righttail}, there exist $C_\alpha$ depending only on $\alpha$ such that
\begin{equation}\label{eq-tail-Y}
\P(Y_{v, N} \geq m_{\bar L} + \lambda) \leq C_\alpha \lambda e^{-\sqrt{2d} \lambda}  e^{-C_\alpha^{-1} \lambda^2/\bar \ell}\mbox{ for all } \lambda \geq 1\,.
\end{equation}

When estimating the ratio $\frac{\Lambda_{\bar N, z}}{\Gamma_{\bar N, z}}$, it is clear that $\frac{\Lambda_{\bar N, z}}{\Gamma_{\bar N, z}} = \frac{\P(E_{v, N}(z))}{\P(F_{v, N}(z))}$ for any fixed $v\in \Xi_{\bar N}$, where the latter concerns only the associated Brownian motion to $X_{v, N}$ and the random variable $Y_{v, N}$. As such, the arguments in \cite[Lemma 4.10]{BDZ13} carry out with merely notation change and give that
\begin{equation}\label{eq-Lambda-Gamma}
\lim_{z\to \infty}\limsup_{\bar L\to \infty} \limsup_{N\to \infty} \frac{\Lambda_{\bar N, z}}{\Gamma_{\bar N, z}} = 1\,.
\end{equation}

Analogous to the proof of \cite[Equation (100)]{BDZ13}, we can compare the field $\{X_{v, N}\}$ to a BRW and apply \cite[Lemma 3.7]{BDZ13} to obtain that
\begin{equation}\label{eq-G}
\P(G_{N}( z)) \leq C_\alpha e^{-\sqrt{2d} z}\,.
\end{equation}
Note that the dimension does not play a significant role in these estimates, as \cite[Lemma 3.7]{BDZ13} follows from a union bound calculation. The dimension changes the volume of the box, but the probability $$\P(X_{v,N}(t) >
z+ \frac{ m_{\bar N}}{\bar n}t  + 10 (\log (t \wedge (n^*-t)))_+ + z^{1/20})$$ scales in the dimension (recall that $m_N$ depends on $d$) which exactly cancels the growth of the volume in $d$. 

The next desired ingredient is the second moment computation for $\Lambda_{\bar N, z}$. Note that (i) our field $\{X_{v, N}: v\in \Xi_{\bar N}\}$ is simply an MBRW (so $\{X_{v, N}\}$ is nicer than its analog in \cite{BDZ13}, which is a sum of an MBRW and a field with uniformly bounded variance); (ii) our  $\{Y_{v, N}\}$ are i.i.d.\ random variable with desired tail bounds as in \eqref{eq-Y-tail} (so also nicer than its analog in \cite{BDZ13}, which has weak correlation for two neighboring local boxes). Therefore, the second moment computation in \cite[Lemma 4.11]{BDZ13} carries out with minimal notation change and gives 
\begin{equation}\label{eq-second-moment-Lambda}
\lim_{z\to \infty} \limsup_{\bar L\to \infty} \limsup_{N \to \infty} \frac{\E (\Lambda_{\bar N, z})^2}{\E \Lambda_{\bar N, z}} = 1\,.
\end{equation}
Note that in \cite[Equation (90)]{BDZ13}, there is no analog of $\limsup_{\bar L\to \infty}$ as in the preceding inequality. That's because we have assumed in \cite{BDZ13} that $L\geq 2^{2^{z^4}}$. Our statement as in \eqref{eq-second-moment-Lambda} is  weaker as it does not give a quantitative dependance on how $\bar L$ should grow in $z$. But this detailed quantitative dependence is not needed for the proof of convergence in law. 

Combining \eqref{eq-lower-tail-fine-field}, \eqref{eq-Lambda-Gamma}, \eqref{eq-G} and \eqref{eq-second-moment-Lambda}, we deduce that
\begin{equation}\label{eq-tail-is-first-moment}
\lim_{z\to \infty} \limsup_{\bar L\to \infty} \limsup_{N\to \infty} \Big| \frac{\P(\max_{v\in V_{\bar N} }\xi_{N, v}^f \geq m_{\bar N} +z)}{\E \Lambda_{\bar N, z}} - 1\Big| = 0\,.
\end{equation}
Therefore, it remains to estimate $\E \Lambda_{\bar N, z}$. To this end, we will follow \cite[Section 4.3]{BDZ13}. We first note that by \eqref{eq-lower-tail-fine-field} and \eqref{eq-tail-is-first-moment}, we have
\begin{equation}\label{eq-Lambda-lower}
\lim_{z\to \infty} \limsup_{\bar L\to \infty} \limsup_{N\to \infty} \frac{\E \Lambda_{\bar N, z}} {z e^{-\sqrt{2d}z}} \geq c_\alpha \,,
\end{equation}
where $c_\alpha>0$ is a constant depending on $\alpha$.

The main goal is to derive the asymototics for $\E \Lambda_{\bar N, z}$.
For $v\in \Xi_{\bar N}$, let $\nu_{v, \bar N}(\cdot)$ be the
density function (of a subprobability measure on $\mathbb R$)
such that, for all $I\subseteq \mathbb{R}$,
$$\int_I\nu_{v, \bar N}(y) dy = \P(X_{v,N}(t) \leq z + \frac{m_{\bar N}}{\bar n}t \mbox{ for all } 0\leq t\leq n^*; X_{v,N}(n^*) - (\bar n-\bar \ell) m_{\bar N}/\bar n \in I)\,.$$
Clearly, by \eqref{eq-Y-tail},
$$\P(E_{v, N}(z)) = \int_{-\infty}^z \nu_{v, \bar N}(y) 
\P(Y_{v, N}\geq \bar \ell m_{\bar N}/\bar n + z - y) dy\,.$$
For a given interval $J$, define
\begin{equation}\label{eq-def-lambda}
\lambda_{v, N, z , J} = \int_{J} \nu_{v, \bar N}(y) 
\P(Y_{v, N}\geq \bar \ell m_{\bar N}/\bar n + z - y) dy\,.
\end{equation}

Set $J_{\bar \ell} = [-\bar \ell, -\bar \ell^{2/5}]$. For convenience of notation, we denote by $A \lesssim B$ that there exists a constant $C_{\alpha}>0$ that depends only on $\alpha$ such that $A\leq C_\alpha B$ for two functions/sequences $A$ and $B$.
As in \cite[Lemma 4.13]{BDZ13}, we claim that 
for any
any sequences $x_{v,N}$ such that 
$|x_{v,N}| \lesssim \bar \ell^{1/5}$,
\begin{equation}\label{eq-limit-J-ell}
\lim_{z\to \infty} \liminf_{\bar \ell \to \infty} \liminf_{N\to \infty} \frac{\sum_{v\in \Xi_N } \lambda_{v, N, z, x_{v, N} +  J_{\bar \ell}}}{\E \Lambda_{N, z}} = 1\,.
\end{equation}
Note that, by containment, the above ratio is always at most $1$.  
We prove \eqref{eq-limit-J-ell} for the case when $x_{v,N} = 0$; the general case follows in the same manner.
Application of the reflection principle (c.f. \cite[Equation (28)]{BDZ13}) to the
Brownian motion with drift,
$\bar X_{v,N}(\cdot)=X_{v,N}(\cdot)-m_{\bar N}t/\bar n$, together with the change of measure
that removes the drift $m_{\bar N} t/\bar n$,  implies that
$$\nu_{v,\bar N}(y)\lesssim e^{-\sqrt{2d} y}2^{-dn^*} z  |y| \,,$$
for $y\le - \bar \ell$, over the given range $z\in (0, \bar \ell)$ (which implies 
$z - y\asymp
|y|$).
Together with \eqref{eq-tail-Y} and independence among $Y_{v, N}$ for $v\in \Xi_{\bar N}$, this implies
the  crude bound
$$\int_{-\infty}^{-\ell}   \nu_{v, \bar N}(y) 
\P(Y_{v, N}\geq \bar \ell m_{\bar N}/\bar n + z - y) dy \lesssim 2^{-d n^*} \mathrm{e}^{-C_\alpha^{-1} \bar \ell }$$
for a constant $C_\alpha>0$ depending on $\alpha$. Similarly,
for $y\leq z$ (and therefore, for $z-y\geq 0$),
application of the reflection principle and
\eqref{eq-tail-Y} again implies that
$$\int_{-\bar \ell^{2/5}}^{z} \nu_{v, \bar N}(y) 
\P(Y_{v, N}\geq \bar \ell m_{\bar N}/\bar n + z - y) dy \lesssim 2^{-d n^*} \bar \ell^{-3/10} z \mathrm{e}^{-\sqrt{2d} z}\,.$$
Together with
\eqref{eq-Lambda-lower},
this completes the verification of \eqref{eq-limit-J-ell}.

Next, we claim that
there exists $\Lambda^*_{K', L', z}>0$ that does not depend on $N$
such that, 
\begin{equation}
	\label{eq-Lambda-KL}
	\lim_{z\to \infty}\limsup_{\bar L\to \infty}\limsup_{N\to \infty} \frac{\E \Lambda_{N, z}}{\Lambda^*_{K', L', z} }= \lim_{z\to \infty}\liminf_{\bar L\to \infty}\liminf_{N\to \infty} \frac{\E \Lambda_{N, z}}{\Lambda^*_{ K', L',  z}} = 1\,.
\end{equation}

By  the
reflection principle and change of measure, we get that for all $y\in [-\bar \ell, z]$ (see the derivation of \cite[Equation (107)]{BDZ13})
\begin{equation}\label{eq-mu-N-v}
  \nu_{v, \bar N}(y) =  2^{-d n^*} \mathrm{e}^{-\sqrt{2d} y} \frac{z(z-y)}{\sqrt{2\pi \log 2}} (1+ O(\bar \ell^3/\bar n)) \,.
\end{equation}
Therefore, 
\begin{align*}
\sum_{v\in \Xi_{\bar N}} \lambda_{v, N, z, J_{\bar \ell}} &= (\frac{\bar N}{\bar L})^d \int_{J_{\bar \ell}}\nu_{v_0, \bar N}(y + O(\bar  \ell/\sqrt{\bar n})) \P(Y_{v_0, N} \geq \sqrt{2d} \log 2 \cdot \bar \ell  + z - y) dy\\
&= 
(1+ O(\bar \ell^3/\sqrt{\bar n}))
\int_{J_{\bar \ell}}  \frac{z(z-y)}{\sqrt{2\pi \log 2}  \mathrm{e}^{\sqrt{2d }y} }  \P(Y_{v_0, N} \geq \sqrt{2d} \log 2 \cdot \bar \ell  + z - y)dy\,,
\end{align*}
where $v_0$ is any fixed vertex in $\Xi_{\bar N}$ and in the last step we have used the fact that $n^* =   \bar n -\bar \ell$. Recall that the law of $Y_{v_0, N}$ is the same as $\max_{u\in V_{\bar L}} \varphi_{\bar L, u} + a_{K', L', u} \phi$, which does not depend on $N$. Combined with \eqref{eq-limit-J-ell}, this completes the proof of \eqref{eq-Lambda-KL}.

Finally, we analyze how $\E \Lambda_{N, z}$ scales with $z$. 
To this end, consider $z_1 < z_2$.  For $v\in \Xi_N$ and $j=1,2$, recall  that
$$\lambda_{v, N,  z_j, z_j+J_{\bar \ell}} = \int_{J_{\bar\ell} + z_j}\nu_{v, \bar N}(y) \P(Y_{v,N}\geq \ell m_{\bar N}/\bar n + z_i - y) dy\,.$$
By \eqref{eq-mu-N-v}, for any $y\in J_{\bar \ell}$ and $z_1, z_2 \ll \log \bar \ell$,
\begin{align*}
\frac{\nu_ {v, \bar N}(y + z_1) \P(Y_{v,N}\geq \bar \ell m_{\bar N}/\bar n - y)}{\nu_{v, \bar N}(y+z_2) \P( Y_{v,N}\geq \bar \ell m_{\bar N}/\bar n - y)} &= \frac{\nu_{v, \bar N}(y+z_1)}{\nu_{v, \bar N}(y+z_2)} = 
(1+ O(\bar \ell^3/\bar n))
\frac{z_1 (z_1 - y)}{z_2 (z_2 - y)} \mathrm{e}^{-\sqrt{2\pi}(z_1-z_2)} \\
&=
(1+ O(\bar \ell^3/\bar n))
\frac{z_1}{z_2} \mathrm{e}^{-\sqrt{2 d}(z_1 - z_2)} (1 + z_2^{-3/5})\,.
\end{align*}
This implies that
$$\frac{\lambda_{v, N, z_1, z_1 + J_{\bar \ell}}}{\lambda_{v, N, z_2, z_2 + J_{\bar \ell}}} 
=(1+ O(\bar \ell^3/\bar n))
 \frac{z_1}{z_2} \mathrm{e}^{-\sqrt{2\pi}(z_1 - z_2)} (1 + z_2^{-3/5})\,.$$
Together with \eqref{eq-limit-J-ell}, the above display implies that
$$\lim_{z_1, z_2\to \infty} \limsup_{\bar L\to \infty}\limsup_{N \to \infty} \frac{z_2e^{-\sqrt{2\pi}z_2}
\E
\Lambda_{N, z_1}}{z_1e^{-\sqrt{2\pi}z_1}\E
\Lambda_{N, z_2}} = \lim_{z_1, z_2\to \infty} \liminf_{\bar L\to \infty} \liminf_{N \to \infty} 
\frac{z_2e^{-\sqrt{2\pi}z_2} \E
\Lambda_{N, z_1}}{
	z_1e^{-\sqrt{2\pi}z_1}	\E\Lambda_{
N, z_2}}  =  1\,.$$
Along with \eqref{eq-Lambda-KL}, this
completes the proof of \eqref{eq-righttail-finefield} for some $\beta^*_{ K', L'}$. From \eqref{eq-lower-tail-fine-field} and \eqref{eq-upper-tail-fine-field}, we see that $c_\alpha\leq \beta^*_{K', L'} \leq C_\alpha$ for all $K', L'$. This completes the proof of the proposition.  \qed

\end{document}